\documentclass[reqno,twoside,11pt,english]{amsart}
\usepackage{amsmath,amsfonts,amssymb,amsthm,epsfig}

\usepackage{color}
\usepackage[final,allcolors=blue,colorlinks=true]{hyperref}

\def\R{\mathbb R}
\def\N{\mathbb N}

\def\al{\alpha}
\def\be{\beta}
\def\ga{\gamma}
\def\de{\delta}
\def\ep{\epsilon}
\def\la{\lambda}
\def\si{\sigma}

\def\var{\varphi}
\def\na{\nabla}
\def\Om{\Omega}  
\def\De{\Delta}      

\def\cal{\mathcal}
\def\L{\mathcal L}                                       
\def\wq{\infty}
\def\pa{\partial}
\def\loc{\text{\rm loc}}
\def\rad{\text{\rm rad}}
\def\Ker{\text{ Ker}}
\def\span{\text{span}}

\newcommand{\D}{{\rm d}}
\newcommand{\wto}{\rightharpoonup}                

\numberwithin{equation}{section}
\newtheorem{theorem}{Theorem}[section]

\newtheorem{corollary}[theorem]{Corollary}

\newtheorem{lemma}[theorem]{Lemma}

\newtheorem{proposition}[theorem]{Proposition}

\theoremstyle{definition}

\makeatother

\usepackage{babel}
\begin{document}
\title[on Kirchhoff equations]{Uniqueness and Nondegeneracy of positive solutions to Kirchhoff equations and its applications in singular perturbation problems}

     \author[G. Li, P. Luo,  S. Peng, C. Wang and C.-L. Xiang]{Gongbao Li, Peng Luo, Shuangjie Peng, Chunhua Wang and Chang-Lin Xiang}

\address[Gongbao Li]{School of Mathematics and Statistics and Hubei Key Laboratory of Mathematical Sciences,  Central China Normal University, Wuhan,  430079, P.R. China}
\email[]{ligb@mail.ccnu.edu.cn}
\address[Peng Luo]{School of Mathematics and Statistics and Hubei Key Laboratory of Mathematical Sciences,  Central China Normal University, Wuhan,  430079, P.R. China}
\email[]{pluo@mail.ccnu.edu.cn}
\address[Shuangjie Peng]{School of Mathematics and Statistics and Hubei Key Laboratory of Mathematical Sciences, Central China Normal University,  Wuhan,  430079, P.R. China}
\email[]{sjpeng@mail.ccnu.edu.cn}
\address[Chunhua Wang]{School of Mathematics and Statistics and Hubei Key Laboratory of Mathematical Sciences, Central China Normal University,  Wuhan,  430079, P.R. China}
\email[]{chunhuawang@mail.ccnu.edu.cn}
\address[Chang-Lin Xiang]{School of Information and Mathematics, Yangtze University, Jingzhou 434023, P.R. China,  and
University of Jyvaskyla, Department of Mathematics and Statistics, P.O. Box 35, FI-40014 University of Jyvaskyla, Finland}
\email[]{xiang\_math@126.com}
\thanks{Corresponding author: Chang-Lin Xiang}
\thanks{Li was supported by  NSFC (No. 11371159), and Program for Changjiang Scholars and Innovative Research Team in University \# IRT13066. Luo and Wang are partially supported by  self-determined research funds of CCNU from colleges' basic research and operation of MOE(CCNU16A05011, CCNU17QN0008). Peng and Wang are also financially supported by NSFC (No. 11571130, No.11671162). Xiang is partially  financially supported by  the Academy of Finland, project 259224.}

\begin{abstract}
In the present paper, we establish the uniqueness and nondegeneracy of positive energy solutions to the  Kirchhoff equation \begin{eqnarray*} -\left(a+b\int_{\mathbb{R}^{3}}|\nabla u|^{2}\right)\Delta  u+u=|u|^{p-1}u & & \text{in }\mathbb{R}^{3}, \end{eqnarray*} where $a,b>0$, $1<p<5$ are constants. Then, as applications, we derive the existence and local uniqueness of solutions to  the  perturbed Kirchhoff  problem
 \begin{eqnarray*} -\left(\epsilon^2a+\epsilon b\int_{\mathbb{R}^{3}}|\nabla u|^{2}\right)\Delta u+V(x)u=|u|^{p-1}u &  & \text{in }\mathbb{R}^{3} \end{eqnarray*} for $\epsilon>0$ sufficiently small, under  some mild assumptions on the potential function $V:\mathbb{R}^3\to \mathbb{R}$.  The existence result is obtained by applying the  Lyapunov-Schmidt reduction method. It seems to be  the first time to study singularly perturbed Kirchhoff problems by reduction method, as all the previous results were obtained by various variational methods. Another advantage of this  approach is that it gives a unified proof to the perturbation problem for all $p\in (1,5)$, which is quite different from using  variational methods in the literature. The local uniqueness result is totally new. It is obtained by using a  type of local Pohozaev identity, which is  developed quite recently by Deng, Lin and Yan in their work ``On the prescribed scalar curvature problem in $\mathbb{R}^N$, local uniqueness and periodicity."  (see J. Math. Pures Appl. (9) {\bf 104}(2015), 1013-1044).
\end{abstract}

\maketitle

{\small
\smallskip
\newline
\tableofcontents}
\bigskip

\section{Introduction and main results}

\subsection{Introduction}

Let $a,b>0$ and $1<p<5$. In this paper, we are concerned with the
following equation
\begin{eqnarray}
-\left(a+b\int_{\R^{3}}|\na u|^{2}\right)\De u+u=u^{p}, & u>0 & \text{in }\R^{3}\label{eq: limiting Kirchhoff}
\end{eqnarray}
and the related perturbation problem
\begin{eqnarray}
-\left(\ep^{2}a+\ep b\int_{\R^{3}}|\na u|^{2}\right)\De u+V(x)u=u^{p}, & u>0 & \text{in }\R^{3},\label{eq: Kirchhoff}
\end{eqnarray}
where $\ep>0$ is a parameter, $V:\R^{3}\to\R$ is a bounded continuous
function.

Problems (\ref{eq: limiting Kirchhoff}), (\ref{eq: Kirchhoff}) and
their variants have been studied extensively in the literature. It
was the physician Kirchhoff \cite{Kirchhoff-1883} that proposed the
following time dependent wave equation
\[
\rho\frac{\pa^{2}u}{\pa t^{2}}-\left(\frac{P_{0}}{h}+\frac{E}{2L}\int_{0}^{L}\left|\frac{\pa u}{\pa x}\right|^{2}\right)\frac{\pa^{2}u}{\pa x^{2}}=0
\]
for the first time, in order to extend the classical D'Alembert's
wave equations for free vibration of elastic strings. Bernstein \cite{Bernstein-1940}
and Pohozaev \cite{Pohozaev-1975} are examples of early research
on the study of Kirchhoff equations. Much attention was received until
J.L. Lions \cite{Lions-1978} introducing an abstract functional framework
to this problem. More interesting results can be found in e.g. \cite{Arosio-Panizzi-1996,Cingolani-Lazzo-1997,DAncona-Spagnolo-1992}
and the references therein. From a mathematical point of view, the
interest of studying Kirchhoff equations comes from the nonlocality
of Kirchhoff type equations. For instance, the consideration of the
stationary analogue of Kirchhoff's wave equation leads to the Dirichlet
problem
\begin{equation}
\begin{cases}
-\left(a+b\int_{\Om}|\na u|^{2}\right)\De u=f(x,u) & \text{in }\Om,\\
u=0 & \text{on }\pa\Om,
\end{cases}\label{eq: Kirchhoff eq. on bdd domain}
\end{equation}
 where $\Om\subset\R^{3}$ is a bounded domain, and to equations of
type
\begin{eqnarray}
-\left(a+b\int_{\R^{3}}|\na u|^{2}\right)\De u=f(x,u) &  & \text{in }\R^{3},\label{eq: general Kirchhoff eq.}
\end{eqnarray}
respectively. In the above two problems, $f$ denotes some nonlinear
functions, a typical example of which is given as in Eq. (\ref{eq: limiting Kirchhoff}).
Note that the term $\left(\int|\na u|^{2}\D x\right)\De u$ depends
not only on the pointwise value of $\De u$, but also on the integral
of $|\na u|^{2}$ over the whole domain. In this sense, Eqs. (\ref{eq: limiting Kirchhoff}),
(\ref{eq: Kirchhoff}), (\ref{eq: Kirchhoff eq. on bdd domain}) and
(\ref{eq: general Kirchhoff eq.}) are no longer the usual pointwise
equalities. This new feature brings new mathematical difficulties
that make the study of Kirchhoff type equations particularly interesting.
We refer to e.g. \cite{Perera-Zhang-2006} and to e.g. \cite{Deng-Peng-Shuai-2015,Figueiredo et al-2014,Guo-2015,He-Li-2015,Li-Li-Shi-2012,Li-Ye-2014}
for mathematical researches on Kirchhoff type equations on bounded
domains and in the whole space, respectively.

Eqs. (\ref{eq: limiting Kirchhoff}) and (\ref{eq: Kirchhoff}) are
also closely related to Schr\"odinger equations. Indeed, notice that
when the constant $b$ vanishes, Eqs. (\ref{eq: limiting Kirchhoff})
and (\ref{eq: Kirchhoff}) reduce to the classical Schr\"odinger
equation
\begin{eqnarray}
-\De w+w=w^{p}, & w>0 & \text{in }\R^{3}\label{eq: classical schrodinger}
\end{eqnarray}
and its perturbation problem
\begin{eqnarray*}
-\ep^{2}\De u+V(x)u=u^{p}, & u>0 & \text{in }\R^{3},
\end{eqnarray*}
respectively. They are special cases of
\begin{eqnarray}
-\ep^{2}\De u+V(x)u=u^{q}, & u>0 & \text{in }\R^{n},\label{eq: Schrodinger equations}
\end{eqnarray}
where $1<q$ is subcritical and $n\ge1$. It is known that Eq. (\ref{eq: classical schrodinger})
admits a unique positive solution (up to translations) which is also
nondegenerate (see e.g. \cite{Berestycki-Lions-1983-1,Berestycki-Lions-1983-2,Chang et al-2007,Kwong-1989}).
Based on this uniqueness and nondegeneracy property, Flower and Weinstein
\cite{Flower-Weinstein-1986}, Oh \cite{Oh-1988,Oh-1990} and many
others proved the existence of solutions to Eq. (\ref{eq: Schrodinger equations})
for $\ep>0$ sufficiently small (the so called semiclassical solutions),
by using the Lyapunov-Schmidt reduction method. Their works motivated
us to study the uniqueness and nondegeneracy of positive solutions
to problem (\ref{eq: limiting Kirchhoff}) and its application in
problem (\ref{eq: Kirchhoff}).

Another motivation of this work is due to the fact that up to now
there have no results on local uniqueness of concentrating solutions
to singularly perturbed Kirchhoff equations, while quite many works
have been devoted to the local uniqueness of concentrating solutions
to singularly perturbed Schr\"odinger equations, see e.g. \cite{Cao-Heinz-2003,Cao-Li-Luo-2015,Glangetas-1993,Grossi}
and the references therein. Here, by local uniqueness, it means that
it has only one solution in the given class of solutions. As an example,
Cao, Li and the second-named author of the present paper recently
considered in \cite{Cao-Li-Luo-2015} the Schr\"odinger equation
(\ref{eq: Schrodinger equations}) under the assumptions that $V$
satisfies

(1) $V$ is a bounded $C^{1}$ function and $\inf_{\R^{n}}V>0$;

(2) There exist $m>1$ and $\de>0$ such that
\[
\begin{cases}
V(x)=V(a_{j})+\sum_{i=1}^{n}b_{j,i}|x_{i}-a_{j,i}|^{m}+O(|x-a_{j}|^{m+1}), & x\in B_{\de}(a_{j}),\\
\frac{\pa V}{\pa x_{i}}=mb_{j,i}|x_{i}-a_{j,i}|^{m-2}(x_{i}-a_{j,i})+O(|x-a_{j}|^{m}), & x\in B_{\de}(a_{j}),
\end{cases}
\]
where $x=(x_{1},\cdots,x_{n})\in\R^{n}$, $a_{j}=(a_{j,1},\cdots,a_{j,n})\in\R^{n}$,
$b_{j,i}\in\R$ with $b_{j,i}\neq0$ for each $i=1,\cdots,n$ and
$j=1,\cdots,k$. By introducing new ideas such as a type of local
Pohozaev identity from Deng, Lin and Yan \cite{Deng-Lin-Yan-2015},
they showed the local uniqueness of multi-bump solutions to problem
(\ref{eq: Schrodinger equations}) concentrating at $k$ different
critical points $\{a_{j}\}_{j=1}^{k}$ of the potential $V$. Here,
by concentrating at $\{a_{j}\}_{j=1}^{k}$, it means that if $u_{\ep}$
is a solution to Eq. (\ref{eq: Schrodinger equations}), then for
any $\de>0$, there exist $\ep_{0}>0$, $R>1$, such that $u_{\ep}(x)\le\de$
for all $|x-a_{j}|\ge\ep R$ and $\ep<\ep_{0}$. Throughout their
proof, the nondegeneracy result of Kwong \cite{Kwong-1989} on positive
solutions to Eq. (\ref{eq: classical schrodinger}) plays a fundamental
role. For more local uniqueness results in this respect, see the references
in \cite{Cao-Li-Luo-2015}. Local uniqueness results have important
applications, as was found for the first time by Deng, Lin and Yan
\cite{Deng-Lin-Yan-2015}. Indeed, Deng, Lin and Yan \cite{Deng-Lin-Yan-2015}
considered solutions of a prescribed scalar curvature problem with
infinitely many bubbles. By considering the normalized difference
of two such solutions, and establishing various Pohozaev-type identities,
they proved that solutions with infinitely many bubbles are unique,
which then implied the periodicity of such solutions under the additional
assumption that the prescribed scalar curvature function is periodic
with respect to one or several variables. Guo, Peng and Yan \cite{Guo-Peng-Yan-2015-arXiv}
further extend the results of Deng, Lin and Yan \cite{Deng-Lin-Yan-2015}
to poly-harmonic problems with critical nonlinearity. Due to the fact
that local uniqueness problem for singularly perturbed Kirchhoff equations
is unknown, in the present paper, we also aim to establish this type
of uniqueness results for Eq. (\ref{eq: Kirchhoff}) under suitable
assumptions.

\subsection{Uniqueness and nondegeneracy results}

 It is  known that Eq. (\ref{eq: limiting Kirchhoff}) is the Euler-Lagrange equation of the energy functional $I:H^{1}(\R^{3})\to\R$ defined as \[ I(u)=\frac{1}{2}\int_{\R^{3}}\left(a|\na u|^{2}+u^{2}\right)+\frac{b}{4}\left(\int_{\R^{3}}|\na u|^{2}        \right)^{2}                    -\frac{1}{p+1}\int_{\R^{3}}|u|^{p+1} \] for $u\in H^{1}(\R^{3})$. Thus critical point theories have been devoted to find solutions for Eq. (\ref{eq: limiting Kirchhoff}) and its variants, see e.g. \cite{He-Zou-2012,Li-Ye-2014,Ye-2015} and the references therein. In particular, the existence of positive solutions of Eq. (\ref{eq: limiting Kirchhoff}) was obtained by looking for  the so called ground states, which is defined as follows: Consider the set of solutions to Eq. (\ref{eq: limiting  Kirchhoff}) and denote \begin{equation} m=\inf\left\{ I(v):v\in H^{1}(\R^{3})\text{ is a nontrivial solution to Eq. }(\ref{eq: Kirchhoff})\right\} .\label{eq: least energy} \end{equation} A nontrivial solution $u$ to Eq. (\ref{eq: limiting  Kirchhoff}) is called a \emph{ground state} if $ I(u)=m$.

The following proposition is summarized from the literature for the readers' convenience.
\begin{proposition} \label{prop: Existence}Let $a,b>0$ and $1<p<5$. Let $m$ be the ground state energy defined as in (\ref{eq: least energy}). Then, there exists a positive ground state of (\ref{eq: limiting Kirchhoff}), and  $ m>0$ holds.

Moreover, for any positive solution $u$, there hold

(1) (smoothness) $u\in C^{\wq}(\R^{3})$;

(2) (symmetry) there exists a decreasing function $v:[0,\wq)\to(0,\wq)$ such that $u=v(|\cdot-x_{0}|)$ for a point $x_{0}\in\R^{3}$;

(3) (Asymptotics) For any multiindex $\al\in\N^{n}$, there exist constants $\de_{\al}>0$ and $C_{\al}>0$ such that \begin{eqnarray*} |D^{\al}u(x)|\le C_{\al}e^{-\de_{\al}|x|} &  & \text{for all }x\in\R^{3}. \end{eqnarray*} \end{proposition}
The existence of ground states of equation (\ref{eq: limiting Kirchhoff}) is implied by Proposition \ref{prop: Existence} of Ye \cite{Ye-2015} \footnote{This reference was brought to us by Ye.}, where more general existence results on Kirchhoff type equations in $\R^{3}$ are obtained. In the special cases when $3<p<5$ and $2<p<3$, the existence has also been proved by He and Zou \cite{He-Zou-2012} and Li and Ye \cite{Li-Ye-2014}, respectively. In particular, in the papers  Ye \cite{Ye-2015} and  Li and Ye \cite{Li-Ye-2014}, to apply the Mountain Pass Lemma to find a ground state solution,  quite complicated manifolds were constructed in order to find a bounded Palais-Smale sequence. The fact that $m>0$ follows from Li and Ye \cite[Lemma 2.8]{Li-Ye-2014}, see also Ye \cite{Ye-2015}. Other properties follow easily from the theory of classical Schr\"odinger equations. For applications of Proposition \ref{prop: Existence}, see e.g. He and Zou \cite{He-Zou-2012}, Li and Ye \cite{Li-Ye-2014} and Ye \cite{Ye-2015} and the references therein.

Proposition \ref{prop: Existence} provides a good understanding on
ground states of Eq. (\ref{eq: limiting Kirchhoff}). However, we
are still left an open problem of uniqueness and nondegeneracy of
the ground state. In the literature, there exist several interesting
results in this respect. For instance, there hold uniqueness and nondegeneracy
of positive solutions to the quasilinear Schr\"odinger equation
\begin{eqnarray}
-\Delta u-u\Delta|u|^{2}+u-|u|^{q-1}u=0 &  & \text{in }\R^{n},\label{eq: quasilinear Schrodinger equation}
\end{eqnarray}
see e.g. \cite{Shinji-Masataka-Tatsuya-2016,Selvitella-2015,Xiang-2016},
and for ground states of the fractional Schr\"odinger equations ($0<s<1\le n$)
\begin{eqnarray*}
\left(-\De\right)^{s}w+w=w^{q}, & w>0 & \text{in }\R^{n},
\end{eqnarray*}
see e.g. \cite{Fall-Valdinoci-2014,Frank-Lenzmann-2013,Frank-Lenzman-Silvestre-2016}.
In the above examples, $q$ is an index standing for the nonlinearity
of subcritical growth. For a systematical research on applications
of nondegeneracy of ground states to perturbation problems, we refer
to Ambrosetti and Malchiodi \cite{Ambrosetti-Malchiodi-Book} and
the references therein. Uniqueness and nondegeneracy results also
play an important role in many other problems. It is known that the
uniqueness and nondegeneracy of ground states are of fundamental importance
when one deals with orbital stability or instability of ground states.
It mainly removes the possibility that directions of instability come
from the kernel of the corresponding linearized operator. The uniqueness
and nondegeneracy of ground states also play an important role in
blow-up analysis for the corresponding standing wave solutions in
the corresponding time-dependent equations, see e.g. Frank et al.
\cite{Frank-Lenzmann-2013,Frank-Lenzman-Silvestre-2016} and the references
therein. Thus, as our first result in this paper, we establish

\begin{theorem} \label{prop: LPX-2016} There exists a unique positive
radial solution $U\in H^{1}(\R^{3})$ satisfying
\begin{eqnarray}
-\left(a+b\int_{\R^{3}}|\na U|^{2}\right)\De U+U=U^{p}, & U>0 & \text{in }\R^{3}.\label{eq: limiting equ.}
\end{eqnarray}
Moreover, $U$ is nondegenerate in $H^{1}(\R^{3})$ in the sense that
there holds\textbf{ }
\[
{\rm Ker}\L={\rm span}\left\{ \pa_{x_{1}}U,\pa_{x_{2}}U,\pa_{x_{3}}U\right\} ,
\]
where $\L:L^{2}(\R^{3})\to L^{2}(\R^{3})$ is the linear operator
defined as
\begin{equation}
\L\var=-\left(a+b\int_{\R^{3}}|\na U|^{2}\right)\De\var-2b\left(\int_{\R^{3}}\na U\cdot\na\var\right)\De U+\var-pU^{p-1}\var\label{eq: L plus}
\end{equation}
for all $\var\in L^{2}(\R^{3})$. \end{theorem}

We remark that our proof of Theorem \ref{prop: LPX-2016} also shows
the existence of positive solutions to Eq. (\ref{eq: limiting Kirchhoff}).
Moreover, we obtain an almost explicit expression for the solutions,
from which Proposition \ref{prop: Existence} follows easily. Also,
our proof is unified for all $p$, $1<p<5$, which is quite different
from the variational methods mentioned above.

In the end of this subsection, let us sketch the proof of Theorem
\ref{prop: LPX-2016}. Recall that to deduce the uniqueness and nondegeneracy
for positive solutions to the local Schr\"odinger equations (\ref{eq: Schrodinger equations})
and (\ref{eq: quasilinear Schrodinger equation}), corresponding ordinary
differential equations are used. That is, to consider the ordinary
differential equations
\begin{eqnarray*}
-\left(u_{rr}+\frac{n-1}{r}u_{r}\right)+u(r)-u^{p}(r)=0, &  & r>0,
\end{eqnarray*}
and
\begin{eqnarray*}
-\left(u_{rr}+\frac{n-1}{r}u_{r}\right)-u(r)\left((u^{2})_{rr}+\frac{n-1}{r}(u^{2})_{r}\right)+u(r)-u^{p}(r)=0, &  & r>0,
\end{eqnarray*}
respectively, where $u_{r}$ is the derivative of $u$ with respect
to $r$, see e.g. Kwong \cite{Kwong-1989} and Adachi et al. \cite{Shinji-Masataka-Tatsuya-2016}.
Therefore, to prove the uniqueness part of Theorem \ref{prop: LPX-2016},
it is quite natural to consider the corresponding ordinary differential
equation to Eq. (\ref{eq: limiting Kirchhoff})
\[
-\left(a+b\int_{0}^{\wq}4\pi r^{2}u_{r}^{2}(r)\right)\left(u_{rr}+\frac{2}{r}u_{r}\right)+u(r)-u^{p}(r)=0
\]
for $0<r<\wq$. However, it turns out that this idea is not so applicable
due to the nonlocality of the term $\int_0^{\wq}4\pi r^2 u_r^{ 2}(r)$.
To overcome this difficulty, our key observation is that the quantity
$\int_0^{\wq}4\pi r^2u_r^{ 2}(r)$ is, in fact, independent of the
choice of the positive solution $u$. Hence we conclude that the coefficient
$ a+b\int_0^{\wq}4\pi r^2 u_r^{ 2}(r)$ is just a positive constant
that is independent of the given solution $u$. At this moment, we
are allowed to apply the uniqueness result of Kwong \cite{Kwong-1989}
on positive solutions to Eq. (\ref{eq: classical schrodinger}) to
prove the uniqueness part of Theorem \ref{prop: LPX-2016}.

To prove the nondegeneracy part of Theorem \ref{prop: LPX-2016},
we apply the spherical harmonics to turn the problem into a system
of ordinary differential equations. It turns out that the key is to
show that the problem $\L\var=0$ has only a trivial radial solution.
In other words, the key step is to show that the positive solution
$u$ of Eq. (\ref{eq: limiting Kirchhoff}) is nondegenerate in the
subspace of radial functions of $H^{1}(\R^{3})$. To this end, again
the above observation plays an essential role. To be precise, write
$c=a+b\int_{0}^{\wq}4\pi r^{2}u_{r}^{2}(r)\D r$ and keep in mind
that $c$ is a constant that is independent of $u$. Introduce an
auxiliary operator ${\cal A}_{u}$ associated to $u$ by defining
\[
{\cal A}_{u}\var=-c\De\var+\var-pu^{p-1}\var
\]
for $\var\in L^{2}(\R^{3})$. Then solving the problem $\L\var=0$,
where $\var$ is radial, is equivalent to solving
\[
{\cal A}_{u}\var=2b\left(\int_{\R^{3}}\na u\cdot\na\var\right)\De u.
\]
 Since ${\cal A}_{u}$ is the linearized operator of positive solutions
to Eq. (\ref{eq: classical schrodinger}) up to a constant, the theory
of the nondegeneracy of positive solutions to Eq. (\ref{eq: classical schrodinger})
are applicable, see Proposition \ref{thm: Kwong and Chang} and Proposition
\ref{prop: property of Au} below. Finishing this step, the rest of
the proof is standard. We refer the readers to the proof of Theorem
\ref{prop: LPX-2016} for details.

\subsection{Existence of semiclassical bounded states}

As applications of Theorem \ref{prop: LPX-2016}, we look for solutions
of (\ref{eq: Kirchhoff}) in the Sobolev space $H^{1}(\R^{3})$ for
sufficiently small $\ep$. Following Oh \cite{Oh-1988}, we call the
solutions as semiclassical solutions. We also call such derived solutions
as concentrating solutions since they will concentrate at certain
point of the potential function $V$.

First let us review some known results. He and Zou \cite{He-Zou-2012}
seems to be the first to study singular perturbed Kirchhoff equations.
In their work \cite{He-Zou-2012}, they considered the problem
\begin{eqnarray*}
-\left(\ep^{2}a+\ep b\int_{\R^{3}}|\na u|^{2}\right)\De u+V(x)u=f(u), & u>0 & \text{in }\R^{3},
\end{eqnarray*}
where $V$ is assumed to satisfy the global condition of Rabinowitz
\cite{Rabinowitz-1992-ZAMP}
\begin{equation}
\liminf_{|x|\to\wq}V(x)>\inf_{x\in\R^{3}}V(x)>0,\label{eq: Rabinowitz condition}
\end{equation}
and $f:\R\to\R$ is a nonlinear function with subcritical growth of
type $u^{q}$ for some $3<q<5$. By using variational method, they
proved the existence of multiple positive solutions for $\ep$ sufficiently
small. Among other results, Wang, Tian, Xu and Zhang \cite{Wang-Tian-Xu-Zhang-2012}
established similar results for Kirchhoff equations with critical
growth
\begin{eqnarray*}
-\left(\ep^{2}a+\ep b\int_{\R^{3}}|\na u|^{2}\right)\De u+V(x)u=f(u)+u^{5}, & u>0 & \text{in }\R^{3},
\end{eqnarray*}
by using variational methods as well, where $V$ and $f$ satisfy
similar conditions as that of \cite{He-Zou-2012}. Based on ``penalization
method'', He, Li and Peng \cite{He-Li-Peng-2014} improved an existence
result of Wang, Tian, Xu and Zhang \cite{Wang-Tian-Xu-Zhang-2012}
by allowing that $V$ only satisfies local conditions: there exists
a bounded open set $\Om\subset\R^{3}$ such that
\begin{equation}
\inf_{\Om}V<\inf_{\pa\Om}V.\label{eq: local Rabinowitz condition}
\end{equation}
Later, by introducing new manifold and applying new approximation
method of \cite{Figueiredo et al-2014}, He and Li \cite{He-Li-2015}
proved the existence of solutions for $\ep$ sufficiently small to
the following problem
\begin{eqnarray*}
-\left(\ep^{2}a+\ep b\int_{\R^{3}}|\na u|^{2}\right)\De u+V(x)u=u^{q}+u^{5}, & u>0 & \text{in }\R^{3},
\end{eqnarray*}
with $V$ satisfying the local condition (\ref{eq: local Rabinowitz condition})
and $1<q<3$. Note that one of their innovations of He and Li \cite{He-Li-2015}
is that they assume $1<q<3$, which is not considered before their
paper. This is due to some drawback of variational methods applied
in previous researches. He \cite{He-2016-JDE} further improved the
results of He and Li \cite{He-Li-2015} by considering Kirchhoff problems
with more general nonlinearity.

From the above, we summarize that all existing results on singularly
perturbed Kirchhoff problems mentioned above are obtained by variational
methods. Moreover, to deal with nonlinearity of type $u^{q}$ for
$q$ in different subintervals of $(1,5]$, different variational
methods have to be applied. By Theorem \ref{prop: LPX-2016}, it is
now possible that we apply Lyapunov-Schmidt reduction to study the
perturbed Kirchhoff equation (\ref{eq: Kirchhoff}). Moreover, it
is expected that this approach can deal with problem (\ref{eq: Kirchhoff})
for all $p$, $1<p<5$, in a unified way, as was shown in Theorem
\ref{prop: LPX-2016}. Indeed, we will derive semiclassical solutions
for problem (\ref{eq: Kirchhoff}) by using Lyapunov-Schmidt reduction
for all $p$, $1<p<5$, in a unified way.

To state our following results, let introduce some notations that
will be used throughout the paper. For $\ep>0$ and $y=(y_{1},y_{2},y_{3})\in\R^{3}$,
write
\begin{eqnarray*}
u_{\ep,y}(x)=u((x-y)/\epsilon), &  & x\in\R^{3}.
\end{eqnarray*}
Assume that $V:\R^{3}\to\R$ satisfies the following conditions:

(V1) $V$ is a bounded continuous function with $\inf_{x\in\R^{3}}V>0$;

(V2) There exist $x_{0}\in\R^{3}$ and $r_{0}>0$ such that
\begin{eqnarray*}
V(x_{0})<V(x) &  & \text{for }0<|x-x_{0}|<r_{0},
\end{eqnarray*}
and $V\in C^{\al}(\bar{B}_{r_{0}}(x_{0}))$ for some $0<\al<1$. That
is, $V$ is of $\al$th order H\"older continuity around $x_{0}$.
Without loss of generality, we assume
\begin{eqnarray*}
x_{0}=0,\quad r_{0}=10 & \text{and} & V(x_{0})=1
\end{eqnarray*}
for simplicity.

The assumption (V1) allows us to introduce the inner products
\[
\langle u,v\rangle_{\ep}=\int_{\R^{3}}\left(\ep^{2}a\na u\cdot\na v+V(x)uv\right)
\]
for $u,v\in H^{1}(\R^{3})$. We also write
\[
H_{\ep}=\{u\in H^{1}(\R^{3}):\|u\|_{\ep}\equiv\langle u,u\rangle_{\ep}^{1/2}<\wq\}.
\]
Denote by $U\in H^{1}(\R^{3})$ the unique positive radial solution
to Eq. (\ref{eq: limiting equ.}). $U$ plays the role of a building
block in the procedure of finding solutions. Now we state the existence
result as follows.

\begin{theorem}\label{thm: main result-existence}Let $a,b>0$ and
$1<p<5$. Suppose that $V$ satisfies (V1) and (V2). Then there exists
$\ep_{0}>0$ such that for all $\ep\in(0,\ep_{0})$, problem (\ref{eq: Kirchhoff})
has a solution $u_{\ep}$ of the form
\[
u_{\ep}=U\left(\frac{x-y_{\ep}}{\ep}\right)+\var_{\ep}
\]
with $\var_{\ep}\in H_{\ep}$, satisfying
\[
y_{\ep}\to x_{0},
\]
\[
\|\var_{\ep}\|_{\ep}=o(\ep^{3/2})
\]
as $\ep\to0$.\end{theorem}

We prove Theorem \ref{thm: main result-existence} by using Lyapunov-Schmidt
reduction based on variational methods. It is known that every solution
to Eq. (\ref{eq: Kirchhoff}) is a critical point of the energy functional
$I_{\ep}:H_{\ep}\to\R$, given by
\begin{equation}
I_{\ep}(u)=\frac{1}{2}\|u\|_{\ep}^{2}+\frac{\ep b}{4}\left(\int_{\R^{3}}|\na u|^{2}\right)^{2}-\frac{1}{p+1}\int_{\R^{3}}u{}_{+}^{p+1}\label{eq: variational functional}
\end{equation}
for $u\in H_{\ep}$, where $u_{+}=\max(u,0)$. It is standard to verify
that $I_{\ep}\in C^{2}(H_{\ep})$. So we are left to find a critical
point of $I_{\ep}$. We will follow the scheme of Cao and Peng \cite{Cao-Peng-2009},
and reduce the problem to find a critical point of a finite dimensional
function (see more details in the next section). However, due to the
presence of the nonlocal term $\left(\int_{\R^{3}}|\na u|^{2}\right)\De u$,
it requires more careful estimates on the orders of $\ep$ in the
procedure. In particular, the nonlocal term brings new difficulties
in the higher order remainder term, which is more complicated than
the case of the Schr\"odinger equation (\ref{eq: Schrodinger equations}).

We remark that to establish Theorem \ref{thm: main result-existence},
we can also assume other types of ``critical'' points in the assumption
(V2). However, for simplicity in the present paper, we will restrict
ourselves to the case as assumed in (V2).

\subsection{Uniqueness of semiclassical bounded states}

Now we state the local uniqueness result. We need the following additional
assumption on $V$:

(V3) $V\in C^{1}(\R^{3})$ and there exist $m>1$ and $\de>0$ such
that
\begin{equation}
\begin{cases}
V(x)=V(x_{0})+\sum_{i=1}^{3}c_{i}|x_{i}-x_{0,i}|^{m}+O(|x-x_{0}|^{m+1}), & x\in B_{\de}(x_{0}),\\
\frac{\pa V}{\pa x_{i}}=mc_{i}|x_{i}-x_{0,i}|^{m-2}(x_{i}-x_{0,i})+O(|x-x_{0}|^{m}), & x\in B_{\de}(x_{0}),
\end{cases}\label{eq: assumptin v3}
\end{equation}
where $c_{i}\in\R$ and $c_{i}\neq0$ for $i=1,2,3$.

\begin{theorem} \label{thm: uniqueness}Assume that $V$ satisfies
(V1), (V2) and (V3). If $u_{\ep}^{(i)}$, $i=1,2$, are two solutions
derived as in Theorem \ref{thm: main result-existence}, then
\[
u_{\ep}^{(1)}\equiv u_{_{\ep}}^{(2)}
\]
holds for $\ep$ sufficiently small.

Moreover, let $u_{\ep}=U_{\ep,y_{\ep}}+\var_{\ep}$ be the unique
solution, then there hold
\[
|y_{\ep}-x_{0}|=o(\ep),
\]
\[
\|\var_{\ep}\|_{\ep}=O(\ep^{3/2+m(1-\tau)})
\]
for some $0<\tau<1$ sufficiently small. \end{theorem}

We remark that if $V$ satisfies (V1), (V2) and (V3), then we must
have $c_{i}>0$ for each $i=1,2,3$ in (\ref{eq: assumptin v3}).
In fact, the assumption (V2) in Theorem \ref{thm: uniqueness} is
only for the use of the existence result of Theorem \ref{thm: main result-existence}.
The arguments of Theorem \ref{thm: uniqueness} show that we can replace
(V2) by working in the class of solutions that satisfy some properties
implied by Theorem \ref{thm: main result-existence}. In this way,
the coefficients $c_{i}$ in (\ref{eq: assumptin v3}) are allowed
to have different signs. For the sake of brevity we only present Theorem
\ref{thm: uniqueness} here, but leave the more general local uniqueness
result in Section \ref{sec: Proof-of-uniqueness-Theorem } (see Theorem
\ref{thm: generalized uniqueness result}).

To prove Theorem \ref{thm: uniqueness}, we will follow the idea of
Cao, Li and Luo \cite{Cao-Li-Luo-2015}. More precisely, if $u_{\ep}^{(i)}$,
$i=1,2$, are two distinct solutions derived as in Theorem \ref{thm: main result-existence},
then it is clear that the function
\[
\xi_{\ep}=(u_{\ep}^{(1)}-u_{\ep}^{(2)})/\|u_{\ep}^{(1)}-u_{\ep}^{(2)}\|_{L^{\wq}(\R^{3})}
\]
satisfies $\|\xi_{\ep}\|_{L^{\wq}(\R^{3})}=1$. We will show, by using
the equations satisfied by $\xi_{\ep}$, that $\|\xi_{\ep}\|_{L^{\wq}(\R^{3})}\to0$
as $\ep\to0$. This gives a contradiction, and thus follows the uniqueness.
To deduce the contradiction, we will need quite delicate estimates
on the asymptotic behaviors of solutions and the concentrating point
$y_{\ep}$. A main tool is a local Pohozaev type identity (see (\ref{prop: Pohozaev identity})).
Again, due to the presence of the nonlocal term $\left(\int_{\R^{3}}|\na u|^{2}\right)\De u$,
the local Pohozaev identity is more complicated than the case of the
Schr\"odinger equation (\ref{eq: Schrodinger equations}). More careful
analysis in the procedure are needed.

Before closing this subsection, let us point out that either in the
literature cited as above or in the present work, solutions to Eq.
(\ref{eq: Kirchhoff}) are of single peak. That is, solutions concentrate
at only one strict local minima of the potential $V$ with only one
peak. It has been known for long time that singularly perturbed Schr\"odinger
equations have multi-peak solutions concentrated at one or more critical
points of $V$ (see e.g. Oh \cite{Oh-1990} and Noussair and Yan \cite{Noussair-Yan-2000}).
However, it seems that there have no results on multi-peak solutions
of singularly perturbed Kirchhoff equations. Thus, a natural question
to be considered after this work is to construct multi-peak solutions
for problem (\ref{eq: Kirchhoff}) under suitable conditions on the
potential $V$, and to show that such constructions are locally unique
as well. We will explore this problem in the forthcoming paper.

\subsection{Organization of the paper and notations}

The paper is organized as follows. In section \ref{sec: Proof-of-Theorem 1.1},
we prove Theorem \ref{prop: LPX-2016}. Then, in Section 3 we give
some preliminaries that will be used for the applications of Theorem
\ref{prop: LPX-2016} later. In section \ref{sec: proof of existence}
we first reduce the problem of finding a critical point for $I_{\ep}$
to that of a finite dimensional function, and then complete the proof
of Theorem \ref{thm: main result-existence}. In section 5, we further
explore some properties of the solutions derived as Theorem \ref{thm: main result-existence},
and introduce a local Pohozaev type identity for solutions to Eq.
(\ref{eq: Kirchhoff}). In section 6, we prove Theorem \ref{thm: uniqueness}.
For brevity, some elementary but long calculations are left in Appendix
A and Appendix B.

Our notations are standard. Denote $u_{+}=\max(u,0)$ for $u\in\R$.
We use $B_{R}(x)$ (and $\bar{B}_{R}(x)$) to denote open (and close)
balls in $\R^{3}$ centered at $x$ with radius $R$. For any $1\le s\le\infty$,
$L^{s}(\R^{3})$ is the standard Banach space of real-valued Lebesgue
measurable functions. A function $u$ belongs to the Sobolev space
$H^{1}(\R^{3})$ if $u$ and all of its first order weak partial derivatives
belong to $L^{2}(\R^{3})$. We use $H^{-1}(\R^{3})$ to denote the
dual space of $H^{1}(\R^{3})$. For the properties of the Sobolev
functions, we refer to the monograph \cite{Ziemer}. By the usual
abuse of notations, we write $u(x)=u(r)$ with $r=|x|$ whenever $u$
is a radial function in $\R^{3}$. We will use $C$ and $C_{j}$ ($j\in\N$)
to denote various positive constants, and $O(t)$, $o(t)$ to mean
$|O(t)|\le C|t|$ and $o(t)/t\to0$ as $t\to0$, respectively.

\section{Proof of Theorem \ref{prop: LPX-2016}\label{sec: Proof-of-Theorem 1.1}}

In this section we prove Theorem \ref{prop: LPX-2016}. Throughout this section we denote by $Q\in H^{1}(\R^{3})$ the unique positive radial function that satisfies \begin{eqnarray} -\De Q+Q=Q^{p} &  & \text{in }\R^{3}.\label{eq: Kwong} \end{eqnarray}  We refer to e.g. Berestycki and Lions \cite{Berestycki-Lions-1983-1} and Kwong \cite{Kwong-1989} for the existence and uniqueness of $Q$, respectively.

\subsection{Uniqueness}

In this subsection we prove the uniqueness part of Theorem \ref{prop: LPX-2016}.
\begin{proof}[Proof of  Uniqueness] Let $u\in H^{1}(\R^{3})$ be an arbitrary positive solution to Eq. (\ref{eq:  limiting Kirchhoff}). Write $c=a+b\int_{\R^{3}}|\na u|^{2}\D x$ so that $u$ satisfies \begin{eqnarray*} -c\De u+u=u^{p} &  & \text{in }\R^{3}. \end{eqnarray*} Then, it is direct to verify that $u(\sqrt{c}(\cdot-t))$ solves Eq. (\ref{eq: Kwong}) for any $t\in\R^{3}$. Thus, the uniqueness of $Q$ implies that \begin{eqnarray*} u(x)=Q\left(\frac{x-t}{\sqrt{c}}\right), &  & x\in\R^{3}, \end{eqnarray*} for some $t\in\R^{3}$. In particular, we obtain $\int_{\R^{3}}|\na u|^{2}\D x=\sqrt{c}\int_{\R^{3}}|\na Q|^{2}\D x$. Substituting this equality into the definition of $c$ yields \[ c=a+b\|\na Q\|_{2}^{2}\sqrt{c}. \] Since $c>0$, this equation is uniquely solved by \begin{equation} \sqrt{c}=\frac{1}{2}\left(b\|\na Q\|_{2}^{2}+\sqrt{b^{2}\|\na Q\|_{2}^{4}+4a}\right).\label{eq: value of c} \end{equation} As a consequence, we deduce that \[ u(x)=Q\left(\frac{2(x-t)}{b\|\na Q\|_{2}^{2}+\sqrt{b^{2}\|\na Q\|_{2}^{4}+4a}}\right) \] for some $t\in\R^{3}$. At this moment, we can easily conclude that the set \[ {\cal M}=\left\{ Q\left(\frac{2(x-t)}{b\|\na Q\|_{2}^{2}+\sqrt{b^{2}\|\na Q\|_{2}^{4}+4a}}\right):t\in\R^{3}\right\} \] consists of all the positive solutions of Eq. (\ref{eq: Kirchhoff}). This finishes the proof.\end{proof}
Note that (\ref{eq: value of c}) implies that the value of $c$ is independent of the choice of positive solutions. This fact will be used repeatly below.

As a consequence, we point out that the following result can be derived naturally.
\begin{corollary} The ground state energy $m$ is an isolated critical value of $I$. \end{corollary}


\subsection{Nondegeneracy}

In this subsection we prove the nondegeneracy part of Theorem \ref{prop: LPX-2016}. We need the following result.
\begin{proposition} \label{thm: Kwong and Chang} Let $1<p<5$ and let $Q\in H^{1}(\R^{3})$ be the unique positive radial ground state of Eq. (\ref{eq: Kwong}). Define the operator ${\cal A}:L^{2}(\R^{3})\to L^{2}(\R^{3})$ as \[ {\cal A}\var=-\De\var+\var-pQ^{p-1}\var \] for $\var\in L^{2}(\R^{3})$. Then the following hold:

(1) $Q$ is nondegenerate in $H^{1}(\R^{3})$, that is, \[ \Ker{\cal A}=\span\left\{ \pa_{x_{1}}Q,\pa_{x_{2}}Q,\pa_{x_{3}}Q\right\} ; \]

(2) The restriction of ${\cal A}$ on $L_{\rad}^{2}(\R^{3})$ is one-to-one and thus it has an inverse ${\cal A}^{-1}:L_{\rad}^{2}(\R^{3})\to L_{\rad}^{2}(\R^{3})$;

(3) ${\cal A}Q=-(p-1)Q^{p}$ and \[ {\cal A}R=-2Q, \] where $R=\frac{2}{p-1}Q+x\cdot\na Q$.\end{proposition}
For a brief proof of (1), we refer to Chang et al. \cite[Lemma 2.1]{Chang et al-2007} (see also the references therein); (2) is an easy consequence of (1) since $Q$ is radial and ${\rm Ker}{\cal A}\cap L_{\rad}^{2}(\R^{3})=\emptyset$; the last result can be obtained by a direct computation, see also Eq. (2.1) of Chang et al. \cite{Chang et al-2007}.

Next, we introduce an auxiliary operator. Let $u$ be a positive solution of Eq. (\ref{eq: limiting Kirchhoff}). Since Eq. (\ref{eq: limiting Kirchhoff}) is translation invariant, we assume with no loss of generality that $u$ is radially symmetric with respect to the origin. Write $c=a+b\int_{\R^{3}}|\na u|^{2}\D x$. Keep in mind that $c$ is a constant that is independent of the choice of $u$ by (\ref{eq: value of c}). Then $u$ satisfies \begin{eqnarray} -c\De u+u-u^{p}=0 &  & \text{in }\R^{3}.\label{eq: auxiliary equation} \end{eqnarray} Define the auxiliary operator ${\cal A}_{u}:L^{2}(\R^{3})\to L^{2}(\R^{3})$ as \[ {\cal A}_{u}\var=-c\De\var+\var-pu^{p-1}\var \] for $\var\in L^{2}(\R^{3})$. The following result on ${\cal A}_{u}$ follows easily from Proposition \ref{thm: Kwong and Chang}.
\begin{proposition} \label{prop: property of Au} ${\cal A}_{u}$ satisfies the following properties:

(1) The kernel of ${\cal A}_{u}$ is given by \[ \Ker{\cal A}_{u}=\span\left\{ \pa_{x_{1}}u,\pa_{x_{2}}u,\pa_{x_{3}}u\right\} ; \]

(2) The restriction of ${\cal A}_{u}$ on $L_{\rad}^{2}(\R^{3})$ is one-to-one and thus it has an inverse ${\cal A}_{u}^{-1}:L_{\rad}^{2}(\R^{3})\to L_{\rad}^{2}(\R^{3})$;

(3) ${\cal A}_{u}u=-(p-1)u^{p}$ and \[ {\cal A}_{u}S=-2u, \] where $S=\frac{2}{p-1}u+x\cdot\na u$. \end{proposition}
\begin{proof} Apply Proposition \ref{thm: Kwong and Chang} to $\tilde{u}$ defined by $\tilde{u}(x)=u(\sqrt{c}x)=Q(x)$. We leave the details to the interested readers. \end{proof} We will also use the standard spherical harmonics to decompose functions in $H^{j}(\R^{N})$ for $j=0,1$, where $N=3$ (see e.g. Ambrosetti and Malchiodi \cite[Chapter 4]{Ambrosetti-Malchiodi-Book}). So let us introduce some necessary notations for the decomposition. Denote by $\De_{\S^{N-1}}$ the Laplacian-Beltrami operator on the unit $N-1$ dimensional sphere $\S^{N-1}$ in $\R^{N}$. Write \begin{eqnarray*} M_{k}=\frac{(N+k-1)!}{(N-1)!k!}\quad\forall\, k\ge0, & \text{and } & M_{k}=0\quad\forall\, k<0. \end{eqnarray*} Denote by $Y_{k,l}$, $k=0,1,\ldots$ and $1\le l\le M_{k}-M_{k-2}$, the spherical harmonics such that \[ -\De_{\S^{N-1}}Y_{k,l}=\la_{k}Y_{k,l} \] for all $k=0,1,\ldots$ and $1\le l\le M_{k}-M_{k-2}$, where \begin{eqnarray*} \la_{k}=k(N+k-2) &  & \forall\, k\ge0 \end{eqnarray*} is an eigenvalue of $-\De_{\S^{N-1}}$ with multiplicity$M_{k}-M_{k-2}$ for all $k\in\N$. In particular, $\la_{0}=0$ is of multiplicity 1 with $Y_{0,1}=1$, and $\la_{1}=N-1$ is of multiplicity $N$ with $Y_{1,l}=x_{l}/|x|$ for $1\le l\le N$.
Then for any function $v\in H^{j}(\R^{N})$, we have the decomposition \[ v(x)=v(r\Om)=\sum_{k=0}^{\wq}\sum_{l=1}^{M_{k}-M_{k-2}}v_{kl}(r)Y_{kl}(\Om) \] with $r=|x|$ and $\Om=x/|x|$, where \begin{eqnarray*} v_{kl}(r)=\int_{\S^{N-1}}v(r\Om)Y_{kl}(\Om)\D\Om &  & \forall\, k,l\ge0. \end{eqnarray*} Note that $v_{kl}\in H^{j}(\R_{+},r^{N-1}\D r)$ holds for all $k,l\ge0$ since $v\in H^{j}(\R^{N})$.

Now we start the proof of nondeneracy part of Theorem \ref{prop: LPX-2016}. We first prove that $u$ is nondegenerate in $H_{\rad}^{1}(\R^{3})$ (in the sense of the following proposition), which is the key ingredient in the proof.
\begin{proposition} \label{prop: nonexistence of radial solutions} Let $\L$ be defined as in (\ref{eq: L plus}) and let $\var\in H_{\rad}^{1}(\R^{3})$ be such that $\L\var=0$. Then $\var\equiv0$ in $\R^{3}$. \end{proposition} \begin{proof} Let $\var\in H_{\rad}^{1}(\R^{3})$ be such that $\L\var=0$. By virtue of the notations introduced above, we can rewrite the equation $\L\var=0$ as below: \[ {\cal A}_{u}\var=2b\left(\int_{\R^{3}}\na u\cdot\na\var\right)\De u. \] We have to prove that $\var\equiv0$. This is sufficient to show that \begin{equation} \int_{\R^{3}}\na u\cdot\na\var=0,\label{eq: Key observation} \end{equation} since then $\var\in\Ker{\cal A}_{u}\cap L_{\rad}^{2}(\R^{3})$, which implies that $\var\equiv0$ by Proposition \ref{prop: property of Au}.

To deduce (\ref{eq: Key observation}), we proceed as follows. Since $u$ is radial and ${\cal A}_{u}$ is one-to-one on $L_{\rad}^{2}(\R^{3})$ by Proposition \ref{prop: property of Au}, $\var$ satisfies the equivalent equation \[ \var=2b\left(\int_{\R^{3}}\na u\cdot\na\var\right){\cal A}_{u}^{-1}(\De u), \] where ${\cal A}_{u}^{-1}$ is the inverse of ${\cal A}_{u}$ restricted on $L_{\rad}^{2}(\R^{3})$.
Next we compute ${\cal A}_{u}^{-1}(\De u)$. By Eq. (\ref{eq: auxiliary equation}), $\De u=(u-u^{p})/c$. Hence ${\cal A}_{u}^{-1}(\De u)=\left({\cal A}_{u}^{-1}(u)-{\cal A}_{u}^{-1}(u^{p})\right)/c$. Applying Proposition \ref{prop: property of Au} (3), we deduce that \[ {\cal A}_{u}^{-1}(\De u)=\frac{1}{c}\left(-\frac{S}{2}+\frac{u}{p-1}\right)=-\frac{1}{2c}x\cdot\na u, \] where $S$ is defined as in Proposition \ref{prop: property of Au}. Therefore, we obtain \[ \var=-\frac{b}{c}\left(\int_{\R^{3}}\na u\cdot\na\var\right)x\cdot\na u=\left(\int_{\R^{3}}\na u\cdot\na\var\right)\psi,\] with $\psi=-\frac{b}{c}x\cdot\na u$.

Now we can deduce (\ref{eq: Key observation}) from the above formula. Taking gradient on both sides gives \[ \na\var=\left(\int_{\R^{3}}\na u\cdot\na\var\right)\na\psi. \] Multiply $\na u$ on both sides and integrate. We achieve \[ \int_{\R^{3}}\na u\cdot\na\var=\left(\int_{\R^{3}}\na u\cdot\na\var\right)\int_{\R^{3}}\na u\cdot\na\psi. \] A direct computation yields that \[ \int_{\R^{3}}\na u\cdot\na\psi=\frac{b}{2c}\int_{\R^{3}}|\na u|^{2}=\frac{c-a}{2c}<\frac{1}{2}. \] Hence we easily deduce that $\int_{\R^{3}}\na u\cdot\na\var=0$, that is, (\ref{eq: Key observation}) holds. The proof of Proposition \ref{prop: nonexistence of radial solutions} is complete. \end{proof}
With the help of Proposition \ref{prop: nonexistence of radial solutions}, we can now finish the proof of Theorem \ref{prop: LPX-2016}. The procedure is standard, see  e.g. Ambrosetti and Malchiodi \cite[Section 4.2]{Ambrosetti-Malchiodi-Book}. For the readers' convenience, we  give a detailed proof.

\begin{proof}[Proof of  nondengeracy] Let $\var\in H^{1}(\R^{3})$ be such that $\L\var=0$. We have to prove that $\var$ is a linear combination of $\pa_{x_{i}}u$, $i=1,2,3$. The idea is to turn the problem $\L\var=0$ into a system of ordinary differential equations by making use of the spherical harmonics to decompose $\var$ into \[ \var=\sum_{k=0}^{\wq}\sum_{l=1}^{M_{k}-M_{k-2}}\var_{kl}(r)Y_{kl}(\Om) \] with $r=|x|$ and $\Om=x/|x|$, where \begin{eqnarray} \var_{kl}(r)=\int_{\S^{2}}\var(r\Om)Y_{kl}(\Om)\D\Om &  & \forall\, k\ge0.\label{eq: kl-th component} \end{eqnarray} Note that $\var_{kl}\in H^{1}(\R_{+},r^{2}\D r)$ holds for all $k,l\ge0$ since $\var\in H^{1}(\R^{3})$.

Combining the fact that  $\int_{\S^2}Y_{kl}\D\si=0$ hold for all $k,l\ge1$,  together with the fact that $u$ is radial, we deduce \[ \int_{\R^{3}}\na u\cdot\na\var=\int_{\R^{3}}\left(-\De u\right)\var=\int_{\R^{3}}\na u\cdot\na\var_{0}, \] where $\var_0(x)=\var_{0,1}(|x|)$ for $x\in\R^3$.
 Hence, the problem $\L\var=0$ is equivalent to the following system of ordinary differential equations:
For $k=0$, we have \begin{equation} \L\var_{0}=0.\label{eq: first case} \end{equation}

For $k=1$, we have \begin{equation} A_1(\var_{1l})\equiv\left(-c\De_{r}+\frac{\la_{1}}{r^{2}}\right)\var_{1l}+\var_{1l}-pu^{p-1}\var_{1l}=0\label{eq: second case} \end{equation} for $l=1,2,3$. Here $\De_r=\pa_{rr}+\frac 2r \pa_r$. We also used the fact that $u$ and $\De u$ are radial functions.

For $k\ge2$, we have that \begin{equation} A_k(\var_{kl})\equiv\left(-c\De_{r}+\frac{\la_{k}}{r^{2}}\right)\var_{kl}+\var_{kl}-pu^{p-1}\var_{kl}=0.\label{eq: third case} \end{equation}

To solve  Eq. (\ref{eq: first case}), we apply Proposition \ref{prop: nonexistence of radial solutions} to conclude that $\var_{0}\equiv0$.

To solve  Eq. (\ref{eq: second case}),  note  that $u^{\prime}$ is a solution of Eq. (\ref{eq: second case})  and $u^{\prime}\in H^{1}(\R_{+},r^{2})$. Since  Eq. (\ref{eq: second case}) is a  second order linear ordinary differential equation, we assume that it has another solution $v(r)=h(r)u^{\prime}(r)$ for some $h$. It is easy to find that $h$ satisfies $$h^{\prime\prime}u^{\prime}+\frac {2}{r}h^{\prime}u^{\prime}+2h(u^{\prime})^{\prime}=0.$$ If $h$ is not identically a constant,  we derive that $$-\frac {h^{\prime\prime}}{h^{\prime}}=2\frac {u^{\prime\prime}}{u^{\prime}}+\frac 2r,$$ which implies that \begin{eqnarray*}  h^{\prime}(r)\sim r^{-2}(u^{\prime})^2&&\text{as}\,\,r\to \infty. \end{eqnarray*} Recall that   $Q=Q(|x|)$, $x\in \R^3$, is the unique positive radial solution of Eq. (\ref{eq: Kwong}). It is well known  that $\lim_{r\to \infty}re^rQ^{\prime}(r)= -C$ holds  for some constant $C>0$. Hence, by the proof of the uniqueness part of  Theorem \ref{prop: LPX-2016}, we know that $\lim_{r\to \infty}re^{r/\sqrt{c}}u^{\prime}(r)=-C_1$ for some $C_1>0$. Combining this fact with the above estimates gives $$|h(r)u^{\prime}(r)|\ge Cr^{-1}e^{r/\sqrt{c}}$$ as $r\to \infty$. Thus $hu^{\prime}$ does not belong to $H^1(\R_+,r^2\D r)$ unless $h$ is a constant. This shows that the family of solutions of  Eq. (\ref{eq: second case}) in  $H^1(\R_+,r^2\D r)$ is given by $hu^{\prime}$, for some constant $h$. In particular, we conclude that $\var_{1l}=d_{l}u^{\prime}$ hold for some constant $d_{l}$, for all $1\le l\le 3$.

For the last Eq.  (\ref{eq: third case}),   we show that it has only a trivial solution. Indeed, for  $k\ge2$, we have \[A_k=A_1+\frac {\de_k}{r^2},\] where $\de_k=\la_k-\la_1$. Since $\la_{k}>\la_{1}$, we find that $\de_k >0$.  Notice that $u^{\prime}$ is an eigenfunction of $A_1$ corresponding to the eigenvalue 0, and that $u^{\prime}$ is of constant sign. By virtue of orthogonality, we can easily infer that $0$ is the smallest eigenvalue of $A_1$. That is, $A_1$ is a nonnegative operator. Therefore, $\de_k>0$ implies that $A_k$ is a  positive operator for all $k\ge 2$. That is, $\langle A_k\psi,\psi \rangle \ge 0$ for all $\psi \in H^1(\R_+,r^2 \D r)$, and the equality attains if and only if $\psi =0$. As a result, we easily prove that if $\var_{kl}$ is a solution of Eq. (\ref{eq: third case}), then $\var_{kl}\equiv0$ holds for all $k\ge 2$.

In summary, we obtain \[ \var=\sum_{l=1}^{3}d_{l}u^{\prime}(r)Y_{1l}=\sum_{l=1}^{3}d_{l}\pa_{x_{l}}u. \] This finishes the proof. \end{proof}

\section{Some preliminaries\label{sec: Some-preliminaries}}

In this section, we explain the strategy of the proof of Theorem \ref{thm: main result-existence}
and present some elementary estimates for later use.

First, denote by $U$ the unique positive radial solution of Eq. (\ref{eq: limiting equ.})
in $H^{1}(\R^{3})$ in the rest of the paper. A useful fact that will
be frequently used is the exponential decay of $U$ of and its derivatives.
That is,
\begin{eqnarray}
U(x)+|\na U(x)|\le Ce^{-\si_{0}|x|}, &  & x\in\R^{3}\label{eq: Properties of U}
\end{eqnarray}
for some $\si_{0}>0$ and $C>0$ (see Proposition \ref{prop: Existence}).

Next, to find solutions for Eq. (\ref{eq: Kirchhoff}) in the form
$U_{\ep,y}+\var$, we introduce a new functional $J_{\ep}:\R^{3}\times H_{\ep}\to\R$
defined by
\begin{eqnarray*}
J_{\ep}(y,\var)=I_{\ep}(U_{\ep,y}+\var), &  & \var\in H_{\ep}.
\end{eqnarray*}
See (\ref{eq: variational functional}) for the definition of $I_{\ep}$.
Following the scheme of Cao and Peng \cite{Cao-Peng-2009}, we divide
the proof of Theorem \ref{thm: main result-existence} into two steps:

\textbf{Step 1:} for each $\ep,\de$ sufficiently small and for each
$y\in B_{\de}(0)$, we will find a critical point $\var_{\ep,y}$
for $J_{\ep}(y,\cdot)$ (the function $y\mapsto\var_{\ep,y}$ also
belongs to the class $C^{1}(H_{\ep})$); then

\textbf{Step 2:} for each $\ep$,$\de$ sufficiently small, we will
find a critical point $y_{\ep}$ for the function $j_{\ep}:B_{\de}(0)\to\R$
induced by
\begin{equation}
y\mapsto j_{\ep}(y)\equiv J(y,\var_{\ep,y}).\label{eq: reduction function}
\end{equation}
That is, we will find a critical point $y_{\ep}$ in the interior
of $B_{\de}(0)$.

It is standard to verify that $(y_{\ep},\var_{\ep,y_{\ep}})$ is a
critical point of $J_{\ep}$ for $\ep$ sufficiently small by the
chain rule. This gives a solution $u_{\ep}\equiv U_{\ep,y_{\ep}}+\var_{\ep,y_{\ep}}$
to Eq. (\ref{eq: Kirchhoff}) for $\ep$ sufficiently small in virtue
of the following lemma.

\begin{lemma}\label{lem: reduction} There exist $\ep_{0}>0$, $\de_{0}>0$
satisfying the following property: for any $\ep\in(0,\ep_{0})$ and
$\de\in(0,\de_{0})$, $y\in B_{\de}(0)$ is a critical point of the
function $j_{\ep}$ define as in (\ref{eq: reduction function}) if
and only if
\[
u_{\ep}\equiv U_{\ep,y}+\var_{\ep,y}
\]
 is a critical point of $I_{\ep}$.\end{lemma}

Lemma \ref{lem: reduction} can be proved in a standard way by using
the arguments as that of Bartsch and Peng \cite{Bartsch-Peng-2007}
(see also Cao and Peng \cite{Cao-Peng-2009}). We leave the details
for the interested readers.

To realize \textbf{Step 1}, expand $J_{\ep}(y,\cdot)$ near $\var=0$
for each fixed $y$:
\[
J_{\ep}(y,\var)=J_{\ep}(y,0)+l_{\ep}(\var)+\frac{1}{2}\langle\L_{\ep}\var,\var\rangle+R_{\ep}(\var),
\]
where $J_{\ep}(y,0)=I_{\ep}(U_{\ep,y})$, and $l_{\ep}$, $\L_{\ep}$
and $R_{\ep}$ are defined as follows: for $\var,\psi\in H_{\ep}$,
define

\begin{equation}
\begin{aligned}l_{\ep}(\var) & =\langle I_{\ep}^{\prime}(U_{\ep,y}),\var\rangle\\
 & =\langle U_{\ep,y},\var\rangle_{\ep}+\ep b\left(\int_{\R^{3}}|\na U_{\ep,y}|^{2}\right)\int_{\R^{3}}\na U_{\ep,y}\cdot\na\var-\int_{\R^{3}}U_{\ep,y}^{p}\var,
\end{aligned}
\label{eq: def of l-epsilon}
\end{equation}
and $\L_{\ep}:L^{2}(\R^{3})\to L^{2}(\R^{3})$ is the bilinear form
around $U_{\ep,y}$ defined by
\begin{equation}
\begin{aligned}\langle\L_{\ep}\var,\psi\rangle & =\langle I_{\ep}^{\prime\prime}(U_{\ep,y})[\var],\psi\rangle\\
 & =\langle\var,\psi\rangle_{\ep}+\ep b\left(\int_{\R^{3}}|\na U_{\ep,y}|^{2}\right)\int_{\R^{3}}\na\var\cdot\na\psi\\
 & \quad+2\ep b\left(\int_{\R^{3}}\na U_{\ep,y}\cdot\na\var\right)\left(\int_{\R^{3}}\na U_{\ep,y}\cdot\na\psi\right)-p\int_{\R^{3}}U_{\ep,y}^{p-1}\var\psi,
\end{aligned}
\label{eq: def of biliear L-epsilon}
\end{equation}
and $R_{\ep}$ denotes the second order reminder term given by
\begin{equation}
R_{\ep}(\var)=J_{\ep}(y,\var)-J_{\ep}(y,0)-l_{\ep}(\var)-\frac{1}{2}\langle\L_{\ep}\var,\var\rangle.\label{eq: 2rd order reminder term}
\end{equation}
We remark that $R_{\ep}$ belongs to $C^{2}(H_{\ep})$ since so is
every term in the right hand side of (\ref{eq: 2rd order reminder term}).

In the rest of this section, we consider $l_{\ep}:H_{\ep}\to\R$ and
$R_{\ep}:H_{\ep}\to\R$ and give some elementary estimates.

We will repeatedly use the following type of Sobolev inequality: for
any $2\le q\le6$ there exists a constant $C>0$ depending only on
$n,V$, $a$ and $q$, but independent of $\ep$, such that
\begin{equation}
\|\var\|_{L^{q}(\R^{3})}\le C\ep^{\frac{3}{q}-\frac{3}{2}}\|\var\|_{\ep}\label{eq: epsilon-Sobolev inequality}
\end{equation}
holds for all $\var\in H_{\ep}$. The proof of (\ref{eq: epsilon-Sobolev inequality})
follows from an elementary scaling argument and the Sobolev embedding
theorems. Indeed, by setting $\tilde{\var}(x)=\var(\ep x)$ and using
Sobolev inequality, we deduce
\[
\begin{aligned}\int_{\R^{3}}|\var|^{q} & =\ep^{3}\int_{\R^{3}}|\tilde{\var}|^{q}\\
 & \le C_{1}\ep^{3}\left(\int_{\R^{3}}\left(|\na\tilde{\var}|^{2}+|\tilde{\var}|^{2}\right)\right)^{q/2}\\
 & =C_{1}\ep^{3-\frac{3q}{2}}\left(\int_{\R^{3}}\left(\ep^{2}|\na\var|^{2}+|\var|^{2}\right)\right)^{q/2}\\
 & \le C_{2}\ep^{3-\frac{3q}{2}}\|\var\|_{\ep}^{q},
\end{aligned}
\]
where $C_{1}$ is the best constant for the Sobolev embedding $H^{1}(\R^{3})\subset L^{q}(\R^{3})$,
and $C_{2}>0$ depends only on $n$, $a,q$ and $V$. We also used
the assumption $\inf_{\R^{3}}V>0$ here.

\begin{lemma}\label{lem: estimate for the first order} Assume that
$V$ satisfies (V1) (V2). Then, there exists a constant $C>0$, independent
of $\ep$, such that for any $y\in B_{1}(0)$, there holds
\[
|l_{\ep}(\var)|\le C\ep^{\frac{3}{2}}\left(\ep^{\al}+\left(V(y)-V(0)\right)\right)\|\var\|_{\ep}
\]
for $\var\in H_{\ep}$. Here $\al$ denotes the order of the H\"older
continuity of $V$ in $B_{10}(0)$.\end{lemma}
\begin{proof}
Since $U$ solves Eq. (\ref{eq: limiting equ.}), we deduce from the
definition (\ref{eq: def of l-epsilon}) of $l_{\ep}$ that
\[
\begin{aligned}l_{\ep}(\var) & =\int_{\R^{3}}(V(x)-V(0))U_{\ep,y}\var\\
 & =\int_{\R^{3}}(V(x)-V(y))U_{\ep,y}\var+(V(y)-V(0))\int_{\R^{3}}U_{\ep,y}\var\\
 & =:l_{1}+l_{2}.
\end{aligned}
\]

To estimate $l_{1}$, we split $l_{1}$ into two parts:
\[
l_{1}=\int_{B_{1}(y)}(V(x)-V(y))U_{\ep,y}\var+\int_{\R^{3}\backslash B_{1}(y)}(V(x)-V(y))U_{\ep,y}\var=:l_{11}+l_{12}.
\]
Combining the assumption (V2) and $y\in B_{1}(0)$, the exponential
decay of $U$ at infinity and (\ref{eq: epsilon-Sobolev inequality}),
we easily derive
\[
|l_{11}|\le C\ep^{\frac{3}{2}+\al}\|\var\|_{\ep}.
\]
By using the boundedness assumption (V1), (\ref{eq: Properties of U}),
applying H\"older's inequality and (\ref{eq: epsilon-Sobolev inequality}),
it yields
\[
|l_{12}|\le C\ep^{\frac{3}{2}+\al}\|\var\|_{\ep}.
\]
Therefore,
\[
|l_{1}|\le|l_{11}|+|l_{12}|\le C\ep^{\frac{3}{2}+\al}\|\var\|_{\ep}.
\]

To estimate $l_{2}$, we use a simple scaling argument and (\ref{eq: epsilon-Sobolev inequality})
to get
\[
|l_{2}|\le C(V(y)-V(0))\ep^{\frac{3}{2}}\|\var\|_{\ep}.
\]
The proof is complete by combining the estimates of $l_{1}$ and $l_{2}$.
\end{proof}
Next we give estimates for $R_{\ep}$ (see (\ref{eq: 2rd order reminder term}))
and its derivatives $R_{\ep}^{(i)}$ for $i=1,2$.

\begin{lemma} \label{lem: error estimates} There exists a constant
$C>0$, independent of $\ep$ and $b$, such that for $i\in\{0,1,2\}$,
there hold
\[
\|R_{\ep}^{(i)}(\var)\|\le C\ep^{-\frac{3(p-1)}{2}}\|\var\|_{\ep}^{p+1-i}+C(b+1)\ep^{-\frac{3}{2}}\left(1+\ep^{-\frac{3}{2}}\|\var\|_{\ep}\right)\|\var\|_{\ep}^{3-i}
\]
 for all $\var\in H_{\ep}$.\end{lemma}

The proof of Lemma \ref{lem: error estimates} is elementary but long.
We leave it in the Appendix B.

\section{The existence of semiclassical solutions\label{sec: proof of existence}}

In this section we prove Theorem \ref{thm: main result-existence}.

\subsection{Finite dimensional reduction}

In this subsection we complete Step 1 as mentioned in Section \ref{sec: Some-preliminaries}.
First we consider the operator $\L_{\ep}$ defined as in (\ref{eq: def of biliear L-epsilon}).
That is,
\[
\begin{aligned}\langle\L_{\ep}\var,\psi\rangle & =\langle\var,\psi\rangle_{\ep}+\ep b\left(\int_{\R^{3}}|\na U_{\ep,y}|^{2}\right)\int_{\R^{3}}\na\var\cdot\na\psi\\
 & \quad+2\ep b\left(\int_{\R^{3}}\na U_{\ep,y}\cdot\na\var\right)\left(\int_{\R^{3}}\na U_{\ep,y}\cdot\na\psi\right)-p\int_{\R^{3}}U_{\ep,y}^{p-1}\var\psi,
\end{aligned}
\]
 for $\var,\psi\in H_{\ep}$. Define
\[
E_{\ep,y}=\left\{ u\in H_{\ep}:\left\langle u,\pa_{y_{i}}U_{\ep,y}\right\rangle _{\ep}=0\text{ for }i=1,2,3\right\} .
\]
When no confuse occurs, we suppress $y$ in the notation $E_{\ep,y}$.
Note that $E_{\ep,y}$ is a closed subspace of $H_{\ep}$ for every
$\ep>0$ and $y\in\R^{3}$. The following result shows that $\L_{\ep}$
is invertible when restricted on $E_{\ep,y}$.

\begin{proposition}\label{prop: bounded inverse operator} There
exist $\ep_{1}>0$, $\de_{1}>0$ and $\rho>0$ sufficiently small,
such that for every $\ep\in(0,\ep_{1})$, $\de\in(0,\de_{1})$, there
holds
\begin{eqnarray*}
\|\L_{\ep}\var\|_{\ep}\ge\rho\|\var\|_{\ep}, &  & \forall\:\var\in E_{\ep,y}
\end{eqnarray*}
uniformly with respect to $y\in B_{\de}(0)$.\end{proposition}
\begin{proof}
We use a contradiction argument. Suppose that there exist $\ep_{n},\de_{n}\to0$,
$y_{n}\in B_{\de_{n}}(0)$ and $\var_{n}\in E_{\ep_{n},y_{n}}$ satisfying
\begin{eqnarray}
\langle\L_{\ep_{n}}\var_{n},g\rangle_{\ep_{n}}\le n^{-1}\|\var_{n}\|_{\ep_{n}}\|g\|_{\ep_{n}}, &  & \forall\:g\in E_{\ep_{n},y_{n}}.\label{eq: contradiction inequa.}
\end{eqnarray}
Since this inequality is homogeneous with respect to $\var_{n}$,
we can assume that
\begin{eqnarray*}
\|\var_{n}\|_{\ep_{n}}^{2}=\ep_{n}^{3} &  & \text{for all }n.
\end{eqnarray*}

Denote $\tilde{\var}_{n}(x)=\var_{n}(\ep_{n}x+y_{n})$. Then
\[
\int_{\R^{3}}\left(a|\na\tilde{\var}_{n}|^{2}+V(\ep_{n}x+y_{n})\tilde{\var}_{n}^{2}\right)=1.
\]
As $V$ is bounded and $\inf_{\R^{3}}V>0$, we infer that $\{\tilde{\var}_{n}\}$
is a bounded sequence in $H^{1}(\R^{3})$. Hence, up to a subsequence,
we may assume that
\begin{eqnarray*}
\tilde{\var}_{n}\wto\var &  & \text{in }H^{1}(\R^{3}),\\
\tilde{\var}_{n}\to\var &  & \text{in }L_{\loc}^{p+1}(\R^{3}),\\
\tilde{\var}_{n}\to\var &  & \text{a.e. in }\R^{3}
\end{eqnarray*}
for some $\var\in H^{1}(\R^{3})$. We will prove that $\var\equiv0$.

First we prove that $\var=\sum_{l=1}^{3}c^{l}\pa_{x_{l}}U$ for some
$c^{l}\in\R$. To this end, let $\tilde{E}_{n}\equiv\left\{ \tilde{g}\in H_{\ep}:\tilde{g}_{\ep_{n},y_{n}}\in E_{\ep_{n},y_{n}}\right\} $,
that is,
\[
\tilde{E}_{n}=\left\{ \tilde{g}\in H_{\ep}:\int_{\R^{3}}\left(a\na\pa_{x_{i}}U\cdot\na\tilde{g}+V(\ep_{n}x+y_{n})\pa_{x_{i}}U\tilde{g}\right)=0\text{ for }i=1,2,3\right\} .
\]
For convenience, denote at the moment
\begin{eqnarray*}
\langle u,v\rangle_{\ast,n}=\int_{\R^{3}}\left(a\na u\cdot\na v+V(\ep_{n}x+y_{n})uv\right) & \text{and} & \|u\|_{\ast,n}^{2}=\langle u,u\rangle_{\ast,n}.
\end{eqnarray*}
Then (\ref{eq: contradiction inequa.}) can be rewritten in terms
of $\tilde{\var}_{n}$ as follows:
\begin{equation}
\begin{aligned} & \quad\langle\tilde{\var}_{n},\tilde{g}\rangle_{\ast,n}+b\int_{\R^{3}}|\na U|^{2}\int_{\R^{3}}\na\tilde{\var}_{n}\cdot\na\tilde{g}\\
 & +2b\int_{\R^{3}}\na U\cdot\na\tilde{\var}_{n}\int_{\R^{3}}\na U\cdot\na\tilde{g}-p\int_{\R^{3}}U^{p-1}\tilde{\var}_{n}\tilde{g}\\
 & \le n^{-1}\|\tilde{g}_{n}\|_{\ast,n},
\end{aligned}
\label{eq: scaled contradiction  inquaiity}
\end{equation}
where $\tilde{g}_{n}(x)=g(\ep_{n}x+y_{n})\in\tilde{E}_{n}$.

Now, for any $g\in C_{0}^{\wq}(\R^{3})$, define $a_{n}^{l}\in\R$
($1\le l\le3$) by
\[
a_{n}^{l}=\langle\pa_{x_{l}}U,g\rangle_{\ast,n}/\|\pa_{x_{l}}U\|_{\ast,n}^{2}
\]
and let $\tilde{g}_{n}=g-\sum_{l=1}^{3}a_{n}^{l}\pa_{x_{l}}U$. Note
that
\[
\|\pa_{x_{l}}U\|_{\ast,n}^{2}\to\int_{\R^{3}}\left(a|\na\pa_{x_{l}}U|^{2}+(\pa_{x_{l}}U)^{2}\right)>0,
\]
and for $l\neq j$,
\[
\langle\pa_{x_{l}}U,\pa_{x_{j}}U\rangle_{\ast,n}=\int_{\R^{3}}V(\ep_{n}x+y_{n})\pa_{x_{l}}U\pa_{x_{j}}U\to\int_{\R^{3}}\pa_{x_{l}}U\pa_{x_{j}}U=0.
\]
Hence the dominated convergence theorem implies that
\[
a_{n}^{l}\to a^{l}\equiv\int_{\R^{3}}\left(a\na\pa_{x_{l}}U\cdot\na g+\pa_{x_{l}}Ug\right)/\int_{\R^{3}}\left(a|\na\pa_{x_{l}}U|^{2}+(\pa_{x_{l}}U)^{2}\right)
\]
and
\[
\langle\pa_{x_{l}}U,\tilde{g}_{n}\rangle_{\ast,n}\to0
\]
as $n\to\wq$. Moreover, we infer that
\[
\|\tilde{g}_{n}\|_{\ast,n}=O(1).
\]
Now substituting $\tilde{g}_{n}$ into (\ref{eq: scaled contradiction  inquaiity})
and letting $n\to\wq$, we find that
\[
\langle\L\var,g\rangle-\sum_{l=1}^{3}a^{l}\langle\L\var,\pa_{x_{l}}U\rangle=0,
\]
where $\L$ is defined as in (\ref{eq: L plus}). Since $U_{x_{l}}\in\text{Ker}\L$
by Theorem \ref{prop: LPX-2016}, we have $\langle\L\var,\pa_{x_{i}}U\rangle=0$.
Thus
\begin{eqnarray*}
\langle\L\var,g\rangle=0, &  & \forall\:g\in C_{0}^{\wq}(\R^{n}).
\end{eqnarray*}
This implies that $\var\in\text{Ker}\L$. Applying Theorem \ref{prop: LPX-2016}
again gives $c^{l}\in\R$ ($1\le l\le3$) such that
\[
\var=\sum_{l=1}^{3}c^{l}\pa_{x_{l}}U.
\]

Next we prove $\var\equiv0$. Note that $\tilde{\var}_{n}\in\tilde{E}_{n}$,
that is,
\[
\int_{\R^{3}}\left(a\na\tilde{\var}_{n}\cdot\na\pa_{x_{l}}U+V(\ep_{n}x+y_{n})\tilde{\var}_{n}\pa_{x_{l}}U\right)=0
\]
for each $l=1,2,3$. By sending $n\to\infty$, we derive

\[
c^{l}\int_{\R^{3}}\left(a|\na\pa_{x_{l}}U|^{2}+(\pa_{x_{l}}U)\right)=0,
\]
which implies $c^{l}=0$. Hence
\begin{eqnarray*}
\var\equiv0 &  & \text{in }\R^{3}.
\end{eqnarray*}

Now we can complete the proof of Proposition \ref{prop: bounded inverse operator}.
We have proved that $\tilde{\var}_{n}\wto0$ in $H^{1}(\R^{3})$ and
$\tilde{\var}_{n}\to0$ in $L_{\loc}^{p+1}(\R^{3})$. As a result
we obtain
\[
\begin{aligned}p\int_{\R^{3}}U_{\ep_{n},y_{n}}^{p-1}\var_{n}^{2} & =p\ep_{n}^{3}\int_{\R^{3}}U^{p-1}\tilde{\var}_{n}^{2}\\
 & =p\ep_{n}^{3}\left(\int_{B_{R}(0)}U^{p-1}\tilde{\var}_{n}^{2}+\int_{\R^{3}\backslash B_{R}(0)}U^{p-1}\tilde{\var}_{n}^{2}\right)\\
 & =p\ep_{n}^{3}\left(o(1)+o_{R}(1)\right),
\end{aligned}
\]
where $o(1)\to0$ as $n\to\wq$ since $\tilde{\var}_{n}\to0$ in $L_{\loc}^{p+1}(\R^{3})$,
and $o_{R}(1)\to0$ as $R\to\wq$ since $\tilde{\var}_{n}\in H^{1}(\R^{3})$
is uniformly bounded. Take $R$ sufficiently large. We get
\[
p\int_{\R^{3}}U_{\ep_{n},y_{n}}^{p-1}\var_{n}^{2}\le\frac{1}{2}\ep_{n}^{3}
\]
for $n$ sufficiently large. However, this implies that
\[
\begin{aligned}\frac{1}{n}\ep_{n}^{3}=\frac{1}{n}\|\var_{n}\|_{\ep_{n}}^{2} & \ge\langle\L_{\ep_{n}}\var_{n},\var_{n}\rangle\\
 & =\|\var_{n}\|_{\ep_{n}}^{2}+b\ep_{n}\int_{\R^{3}}|\na U_{\ep_{n},y_{n}}|^{2}\int_{\R^{3}}|\na\var_{n}|^{2}\\
 & \quad+2b\ep_{n}\left(\int_{\R^{3}}\na U_{\ep_{n},y_{n}}\cdot\na\var_{n}\right)^{2}-p\int_{\R^{3}}U_{\ep_{n},y_{n}}^{p-1}\var_{n}^{2}\\
 & \ge\frac{1}{2}\ep_{n}^{3}.
\end{aligned}
\]
We reach a contradiction. The proof is complete.
\end{proof}
Proposition \ref{prop: bounded inverse operator} implies that by
restricting on $E_{\ep,y}$, the quadratic form $\L_{\ep}:E_{\ep,y}\to E_{\ep,y}$
has a bounded inverse, with $\|\L_{\ep}^{-1}\|\le\rho^{-1}$ uniformly
with respect to $y\in B_{\de}(0)$. This further implies the following
reduction map.

\begin{proposition} \label{prop: reduction map} There exist $\ep_{0}>0$,
$\de_{0}>0$ sufficiently small such that for all $\ep\in(0,\ep_{0})$,
$\de\in(0,\de_{0})$, there exists a $C^{1}$ map $\var_{\ep}:B_{\de}(0)\to H_{\ep}$
with $y\mapsto\var_{\ep,y}\in E_{\ep,y}$ satisfying
\begin{eqnarray*}
\left\langle \frac{\pa J_{\ep}(y,\var_{\ep,y})}{\pa\var},\psi\right\rangle _{\ep}=0, &  & \forall\:\psi\in E_{\ep,y}.
\end{eqnarray*}
Moreover, we can choose $\tau\in(0,\al/2)$ as small as we wish, such
that
\begin{equation}
\|\var_{\ep,y}\|_{\ep}\le\ep^{\frac{3}{2}}\left(\ep^{\al-\tau}+\left(V(y)-V(0)\right)^{1-\tau}\right).\label{eq: a simple estimate}
\end{equation}
\end{proposition}
\begin{proof}
This existence of the mapping $y\mapsto\var_{\ep,y}$ follows from
the contraction mapping theorem. We construct a contraction map as
follows.

Let $\ep_{1}$ and $\de_{1}$ be defined as in Proposition \ref{prop: bounded inverse operator}.
Let $\ep_{0}\le\ep_{1}$ and $\de_{0}\le\de_{1}$. We will choose
$\ep_{0}$ and $\de_{0}$ later. Fix $y\in B_{\de}(0)$ for $\de<\de_{0}$.
Recall that
\[
J_{\ep}(y,\var)=I_{\ep}(U_{\ep,y})+l_{\ep}(\var)+\frac{1}{2}\langle\L_{\ep}\var,\var\rangle+R_{\ep}(\var).
\]
So we have
\[
\frac{\pa J_{\ep}(\var)}{\pa\var}=l_{\ep}+\L_{\ep}\var+R_{\ep}^{\prime}(\var).
\]
Since $E_{\ep,y}$ is a closed subspace of $H_{\ep}$, Lemma \ref{lem: estimate for the first order}
and Lemma \ref{lem: error estimates} implies that $l_{\ep}$ and
$R_{\ep}^{\prime}(\var)$ are bounded linear operators when restricted
on $E_{\ep,y}$. So we can identify $l_{\ep}$ and $R_{\ep}^{\prime}(\var)$
with their representatives in $E_{\ep,y}$. Then, to prove Proposition
\ref{prop: reduction map}, it is equivalent to find $\var\in E_{\ep,y}$
that satisfies
\begin{equation}
\var={\cal A}_{\ep}(\var)\equiv-\L_{\ep}^{-1}\left(l_{\ep}+R_{\ep}^{\prime}(\var)\right).\label{eq: contraction map}
\end{equation}

To solve (\ref{eq: contraction map}), define
\[
N_{\ep}=\left\{ \var\in E_{\ep,y}:\|\var\|_{\ep}\le\ep^{\frac{3}{2}}\left(\ep^{\al-\tau}+\left(V(y)-V(0)\right)^{1-\tau}\right)\right\}
\]
for $\tau\in(0,\al/2)$ sufficiently small. We prove that ${\cal A}_{\ep}:N_{\ep}\to N_{\ep}$
is a contraction map.

First we show that ${\cal A}_{\ep}(N_{\ep})\subset N_{\ep}$. Let
$\rho>0$ be the constant defined as in Proposition \ref{prop: bounded inverse operator}
so that $\|\L_{\ep}^{-1}\|\le\rho^{-1}$. Then for $\var\in N_{\ep}$
we have
\[
|{\cal A}_{\ep}(\var)|\le\rho^{-1}\left(\|l_{\ep}\|+\|R_{\ep}^{\prime}(\var)\|\right).
\]
Lemma \ref{lem: estimate for the first order} gives
\[
\|l_{\ep}\|\le C\ep^{\frac{3}{2}}\left(\ep^{\al}+\left(V(y)-V(0)\right)\right),
\]
where $C$ is independent of $\ep$ and $y$. Since $\tau>0$ and
$V(y)\to V(0)$ as $y\to0$, we may choose $\de_{0}$ and $\ep_{0}$
sufficiently small, so that when $\ep<\ep_{0}$ and $\de<\de_{0}$
we have
\begin{equation}
\|l_{\ep}\|\le\frac{\rho}{2}\ep^{\frac{3}{2}}\left(\ep^{\al-\tau}+\left(V(y)-V(0)\right)^{1-\tau}\right).\label{eq: estimate for the first order term}
\end{equation}
To estimate $\|R_{\ep}^{\prime}(\var)\|$ we use Lemma \ref{lem: error estimates}.
Note that $\var\in N_{\ep}$ implies that
\begin{equation}
\ep^{-\frac{3}{2}}\|\var\|_{\ep}\le\ep^{\al-\tau}+\left(V(y)-V(0)\right)^{1-\tau}=o(1),\label{eq: infinitesimal}
\end{equation}
where $o(1)\to0$ as $\ep_{0},\de_{0}\to0$. In particular, for $\var\in N_{\ep}$
and $\ep$ sufficiently small, we have
\begin{equation}
1+\ep^{-\frac{3}{2}}\|\var\|_{\ep}\le2.\label{eq:  bounded error}
\end{equation}
Hence, using Lemma \ref{lem: error estimates} and (\ref{eq: infinitesimal})
gives
\begin{equation}
\begin{aligned}\|R_{\ep}^{\prime}(\var)\| & \le C\ep^{-\frac{3(p-1)}{2}}\|\var\|_{\ep}^{p}+C(b+1)\ep^{-\frac{3}{2}}\|\var\|_{\ep}^{2}\le\frac{\rho}{2}\|\var\|_{\ep}.\end{aligned}
\label{eq: estimate for the error term}
\end{equation}
Combining (\ref{eq: estimate for the first order term}) and (\ref{eq: estimate for the error term}),
we deduce
\[
{\cal A}_{\ep}(\var)\le\ep^{\frac{3}{2}}\left(\ep^{\al-\tau}+\left(V(y)-V(0)\right)^{1-\tau}\right).
\]
This proves that ${\cal A}_{\ep}(N_{\ep})\subset N_{\ep}$.

Next we prove that ${\cal A}_{\ep}:N_{\ep}\to N_{\ep}$ is a contraction
map. For any $\var,\psi\in N_{\ep}$, we have
\[
|{\cal A}_{\ep}(\var)-{\cal A}_{\ep}(\psi)|=|\L_{\ep}^{-1}(R_{\ep}^{\prime}(\var)-R_{\ep}^{\prime}(\psi))|\le\rho^{-1}|(R_{\ep}^{\prime}(\var)-R_{\ep}^{\prime}(\psi))|.
\]
To estimate $|(R_{\ep}^{\prime}(\var)-R_{\ep}^{\prime}(\psi))|$,
we use the estimate for $R_{\ep}^{\prime\prime}$ in Lemma \ref{lem: error estimates}.
Argue similarly as above. We obtain
\[
\begin{aligned}|(R_{\ep}^{\prime}(\var)-R_{\ep}^{\prime}(\psi))| & \le C\ep^{-\frac{3(p-1)}{2}}\left(\|\var\|_{\ep}^{p-1}+\|\psi\|_{\ep}^{p-1}\right)\|\var-\psi\|_{\ep}\\
 & \quad+C\ep^{-\frac{3}{2}}\left(\|\var\|_{\ep}+\|\psi\|_{\ep}\right)\|\var-\psi\|_{\ep}\\
 & \le\frac{\rho}{2}\|\var-\psi\|_{\ep},
\end{aligned}
\]
where we have used (\ref{eq: infinitesimal}) and (\ref{eq:  bounded error})
again. Hence we deduce that
\[
|{\cal A}_{\ep}(\var)-{\cal A}_{\ep}(\psi)|\le\frac{1}{2}\|\var-\psi\|_{\ep}.
\]
This shows that ${\cal A}_{\ep}:N_{\ep}\to N_{\ep}$ is a contraction
map. Thus, there exists a contraction map $y\mapsto\var_{\ep,y}$
such that (\ref{eq: contraction map}) holds.

At last, we claim that the map $y\mapsto\var_{\ep,y}$ belongs to
$C^{1}$. Indeed, by similar arguments as that of Cao, Noussair and
Yan \cite{Cao-Noussair-Yan-2008}, we can deduce a unique $C^{1}$-map
$\tilde{\var}_{\ep,y}:B_{\de}(0)\to E_{\ep,y}$ which satisfies (\ref{eq: contraction map}).
Therefore, by the uniqueness $\var_{\ep,y}=\tilde{\var}_{\ep,y}$,
and hence the claim follows.
\end{proof}

\subsection{Proof of Theorem \ref{thm: main result-existence}}

First we give the following observation.

\begin{lemma} \label{lem: 4.2}There holds
\[
\left\langle \L_{\ep}\var,\var\right\rangle =O\left(\|\var\|_{\ep}^{2}\right)
\]
for $\var\in E_{\ep}$.\end{lemma}
\begin{proof}
The proof is direct. Recall that
\[
\begin{aligned}\langle\L_{\ep}\var,\var\rangle & =\|\var\|_{\ep}^{2}+b\ep\left(\int_{\R^{3}}|\na U_{\ep,y}|^{2}\int_{\R^{3}}|\na\var|^{2}+2\left(\int_{\R^{3}}\na U_{\ep,y}\cdot\na\var\right)^{2}\right)\\
 & \quad-p\int_{\R^{3}}U_{\ep,y}^{p-1}\var^{2}.
\end{aligned}
\]
Direct computation gives
\[
\ep\int_{\R^{3}}|\na U_{\ep,y}|^{2}\int_{\R^{3}}|\na\var|^{2}=C_{0}\ep^{2}\int_{\R^{3}}|\na\var|^{2}
\]
with $C_{0}=\int_{\R^{3}}|\na U|^{2}$. Thus using this equality and
H\"older's inequality yields
\[
\ep\left(\int_{\R^{3}}\na U_{\ep,y}\cdot\na\var\right)^{2}=O\left(\|\var\|_{\ep}^{2}\right).
\]
The last term $\int_{\R^{3}}U_{\ep,y}^{p-1}\var^{2}$ can be estimated
by using (\ref{eq: epsilon-Sobolev inequality}). Combing the above
estimates together, we complete the proof.
\end{proof}
Now we prove Theorem \ref{thm: main result-existence}.

\begin{proof}[Proof of Theorem \ref{thm: main result-existence}]
Let $\ep_{0}$ and $\de_{0}$ be defined as in Proposition \ref{prop: reduction map}
and let $\ep<\ep_{0}$. Fix $0<\de<\de_{0}$. Let $y\mapsto\var_{\ep,y}$
for $y\in B_{\de}(0)$ be the map obtained in Proposition \ref{prop: reduction map}.
As aforementioned in Step 2 in Section \ref{sec: Some-preliminaries},
it is equivalent to find a critical point for the function $j_{\ep}$
defined as in (\ref{eq: reduction function}) by Lemma \ref{lem: reduction}.
By the Taylor expansion, we have
\[
j_{\ep}(y)=J(y,\var_{\ep,y})=I_{\ep}(U_{\ep,y})+l_{\ep}(\var_{\ep,y})+\frac{1}{2}\langle\L_{\ep}\var_{\ep,y},\var_{\ep,y}\rangle+R_{\ep}(\var_{\ep,y}).
\]
We analyze the asymptotic behavior of $j_{\ep}$ with respect to $\ep$
first.

By Proposition \ref{Prop: asymptotics of perturbation of U}, we have
\[
I_{\ep}(U_{\ep,y})=A\ep^{3}+B\ep^{3}\left(V(y)-V(0)\right)+O(\ep^{3+\al})
\]
for some constants $A,B\in\R$. Lemma \ref{lem: estimate for the first order}
and Proposition \ref{prop: reduction map} give
\[
l_{\ep}(\var_{\ep,y})=O(\ep^{3})\left(\ep^{\al}+\left(V(y)-V(0)\right)\right)\left(\ep^{\al-\tau}+\left(V(y)-V(0)\right)^{1-\tau}\right).
\]
Lemma \ref{lem: 4.2} gives $\langle\L_{\ep}\var_{\ep,y},\var_{\ep,y}\rangle=O(\|\var_{\ep,y}\|_{\ep}^{2})$.
Lemma \ref{lem: error estimates} gives
\[
|R_{\ep}(\var_{\ep,y})|\le C\left(\ep^{-\frac{3(p-1)}{2}}\|\var_{\ep,y}\|_{\ep}^{p+1}+\ep^{-\frac{3}{2}}\|\var_{\ep,y}\|_{\ep}^{3}\right)=o(1)\|\var_{\ep,y}\|_{\ep}^{2},
\]
 where we have used (\ref{eq: infinitesimal}) since $\var_{\ep,y}\in N_{\ep}$.
Combining the above estimates yields
\[
\begin{aligned}j_{\ep}(y) & =A\ep^{3}+B\ep^{3}\left(V(y)-V(0)\right)+O(\ep^{3+\al})\\
 & \quad+O(\ep^{3})\left(\ep^{\al-\tau}+\left(V(y)-V(0)\right)^{1-\tau}\right)^{2}.
\end{aligned}
\]

Now consider the minimizing problem
\[
j_{\ep}(y_{\ep})\equiv\inf_{y\in B_{\de}(0)}j_{\ep}(y).
\]
We claim that $y_{\ep}\in B_{\de}(0)$. That is, $y_{\ep}$ is an
interior point of $B_{\de}(0)$.

To prove the claim, we apply a comparison argument. Let $e\in\R^{3}$
with $|e|=1$ and $\eta>1$. We will choose $\eta$ later. Let $z_{\ep}=\ep^{\eta}e\in B_{\de}(0)$
for a sufficiently large $\eta>1$. By the above asymptotics formula,
we have
\[
\begin{aligned}j_{\ep}(z_{\ep}) & =A\ep^{3}+B\ep^{3}\left(V(z_{\ep})-V(0)\right)+O(\ep^{3+\al})\\
 & \quad+O(\ep^{3})\left(\ep^{\al-\tau}+\left(V(z_{\ep})-V(0)\right)^{1-\tau}\right)^{2}.
\end{aligned}
\]
Applying the H\"older continuity of $V$, we derive that
\[
\begin{aligned}j_{\ep}(z_{\ep}) & =A\ep^{3}+O(\ep^{3+\al\eta})+O(\ep^{3+\al})\\
 & \quad+O(\ep^{3}(\ep^{2(\al-\tau)}+\ep^{2\eta\al(1-\tau)}))\\
 & =A\ep^{3}+O(\ep^{3+\al}),
\end{aligned}
\]
where $\eta>1$ is chosen to be sufficiently large accordingly. Note
that we also used the fact that $\tau\ll\al/2$. Thus, by using $j(y_{\ep})\le j(z_{\ep})$
we deduce
\[
B\ep^{3}\left(V(y_{\ep})-V(0)\right)+O(\ep^{3})\left(\ep^{\al-\tau}+\left(V(y_{\ep})-V(0)\right)^{1-\tau}\right)^{2}\le O(\ep^{3+\al}).
\]
That is,
\begin{equation}
B\left(V(y_{\ep})-V(0)\right)+O(1)\left(\ep^{\al-\tau}+\left(V(y_{\ep})-V(0)\right)^{1-\tau}\right)^{2}\le O(\ep^{\al}).\label{eq: comparision inequality}
\end{equation}
If $y_{\ep}\in\pa B_{\de}(0)$, then by the assumption (V2), we have
\[
V(y_{\ep})-V(0)\ge c_{0}>0
\]
for some constant $0<c_{0}\ll1$ since $V$ is continuous at $x=0$
and $\de$ is sufficiently small. Thus, by noting that $B>0$ from
Proposition \ref{Prop: asymptotics of perturbation of U} and sending
$\ep\to0$, we infer from (\ref{eq: comparision inequality}) that
\[
c_{0}\le0.
\]
We reach a contradiction. This proves the claim. Thus $y_{\ep}$ is
a critical point of $j_{\ep}$ in $B_{\de}(0)$.

Theorem \ref{thm: main result-existence} now follows from the claim
and Lemma \ref{lem: reduction}. \end{proof}

\section{Local Pohozaev identity }

In this section we present some preliminaries for the proof of Theorem
\ref{thm: uniqueness}. In particular, we derive a local Pohozaev
type identity which plays an important role in the proof of Theorem
\ref{thm: uniqueness}.

First we explore some properties of the solutions derived as in Theorem
\ref{thm: main result-existence}. Let $u_{\ep}=U_{\ep,y_{\ep}}+\var_{\ep}$
be a solution of Eq. (\ref{eq: Kirchhoff}). By Theorem \ref{thm: main result-existence},
we know $\|\var_{\ep}\|_{\ep}=O(\ep^{3/2})$. Thus, a straightforward
computation gives
\begin{equation}
\|u_{\ep}\|_{\ep}=O(\ep^{3/2}).\label{eq: size of the solution}
\end{equation}
Set
\[
\bar{u}_{\ep}(x)=u_{\ep}(\ep x+y_{\ep}).
\]
Then $\bar{u}_{\ep}>0$ solves
\begin{eqnarray}
-\left(a+b\int_{\R^{3}}|\na\bar{u}_{\ep}|^{2}\right)\De\bar{u}_{\ep}+\bar{V}_{\ep}(x)\bar{u}_{\ep}=\bar{u}_{\ep}^{p} &  & \text{in }\R^{3},\label{eq: scaled equation}
\end{eqnarray}
with $\bar{V}_{\ep}(x)=V(\ep x+y_{\ep})$. Moreover, there holds
\begin{equation}
\int_{\R^{3}}\left(a|\na\bar{u}_{\ep}|^{2}+\bar{V}_{\ep}\bar{u}_{\ep}^{2}\right)=\ep^{-3}\|u_{\ep}\|_{\ep}^{2}=O(1)\label{eq: uniform boundedness}
\end{equation}
by (\ref{eq: size of the solution}).

By the assumption (V1), $\bar{V}_{\ep}$ is bounded uniformly with
respect to $\ep$, and
\[
\ga\equiv\inf_{x\in\R^{3}}\bar{V}_{\ep}(x)>0.
\]
Therefore, $\bar{u}_{\ep}$ satisfies
\[
\begin{cases}
-a\De\bar{u}_{\ep}+\ga\bar{u}_{\ep}\le\bar{u}_{\ep}^{p} & \text{in }\R^{3},\\
\sup_{\ep}\|\bar{u}_{\ep}\|_{H^{1}(\R^{3})}\le C<\wq.
\end{cases}
\]
Using the comparison principle as that of He and Xiang \cite{He-Xiang},
we infer that
\begin{eqnarray}
\bar{u}_{\ep}(x)\le Ce^{-\eta|x|}, &  & x\in\R^{3}\label{eq: decay estimates}
\end{eqnarray}
holds for some constants $C,\eta>0$ independent of $\ep>0$.

We remark that (\ref{eq: decay estimates}) is equivalent to
\begin{eqnarray*}
u_{\ep}(x)\le Ce^{-\frac{\eta|x-y_{\ep}|}{\ep}}, &  & x\in\R^{3},
\end{eqnarray*}
which means that $u_{\ep}$ concentrates at $x=0$ rapidly as $\ep\to0$.
In particular, under the additional assumption (V3), we will prove
that $y_{\ep}=o(\ep)$, which in turn implies that $u_{\ep}(x)\le Ce^{-\eta|x|/\ep}$
for $x\in\R^{3}$. This shows that the solutions concentrate around
the minima of $V$.

Furthermore, by using the Bessel potential, we derive
\[
\bar{u}_{\ep}\le\frac{1}{-a\De+\ga}\bar{u}_{\ep}^{p}.
\]
Since $p<5$ is $H^{1}$-subcritical, the standard potential theory
and iteration arguments shows that $\bar{u}_{\ep}\in L^{\wq}(\R^{3})$
and
\[
\|\bar{u}_{\ep}\|_{L^{\wq}(\R^{3})}\le C<\wq
\]
holds for some $C>0$ uniformly with respect to $\ep$. As a consequence
of this estimates and the assumption (V1), we further infer from Eq.
(\ref{eq: scaled equation}) that
\begin{equation}
\|\De\bar{u}_{\ep}\|_{L^{\wq}(\R^{3})}\le C\label{eq: boundedness estimate for laplacian}
\end{equation}
holds uniformly with respect to sufficiently small $\ep>0$.

Next we derive a local Pohozaev-type identity for solutions of Eq.
(\ref{eq: Kirchhoff}).

\begin{proposition} \label{prop: Pohozaev identity}Let $u$ be a
positive solution of Eq. (\ref{eq: Kirchhoff}). Let $\Om$ be a bounded
smooth domain in $\R^{3}$. Then, for each $i=1,2,3$, there hold
\begin{equation}
\begin{aligned}\int_{\Om}\frac{\pa V}{\pa x_{i}}u^{2} & =\left(\ep^{2}a+\ep b\int_{\R^{3}}|\na u|^{2}\right)\int_{\pa\Om}\left(|\na u|^{2}\nu_{i}-2\frac{\pa u}{\pa\nu}\frac{\pa u}{\pa x_{i}}\right)\\
 & \quad+\int_{\pa\Om}Vu^{2}\nu_{i}-\frac{2}{p+1}\int_{\pa\Om}u{}^{p+1}\nu_{i}.
\end{aligned}
\label{eq: Pohozaev}
\end{equation}
Here $\nu=(\nu_{1},\nu_{2},\nu_{3})$ is the unit outward normal of
$\pa\Om$.\end{proposition}
\begin{proof}
Identity (\ref{eq: Pohozaev}) is obtained by multiplying $\pa_{x_{i}}u$
on both sides of Eq. (\ref{eq: Kirchhoff}) and then do integrating
by parts. We refer the readers to Proposition 2.3 of Cao, Li and Luo
\cite{Cao-Li-Luo-2015} for details. 
\end{proof}
Now, let $u_{\ep}=U_{\ep,y_{\ep}}+\var_{\ep,y_{\ep}}$ be an arbitrary
solution of Eq. (\ref{eq: Kirchhoff}) derived as in Theorem \ref{thm: main result-existence}.
We know $y_{\ep}=o(1)$ as $\ep\to0$ from Theorem \ref{thm: main result-existence}.
We will improve this asymptotics estimate by assuming that $V$ satisfies
the additional assumption (V3), and by means of the above Pohozaev
type identity. However, before we proceed further, let us give some
observations first.

Notice that using polar coordinates, there holds
\[
\int_{1}^{2}\int_{\pa B_{r}(y_{\ep})}|f|=\int_{\{1<|x-y_{\ep}|<2\}}|f|\le\int_{\R^{3}}|f|
\]
for any $f\in L^{1}(\R^{3})$. So, there exists $d\in(1,2)$, possibly
depending on $f$, such that
\[
\int_{\pa B_{d}(y_{\ep})}|f|\le\int_{\R^{3}}|f|.
\]
Applying this inequality to $f=\ep^{2}|\na\var_{\ep}|^{2}+\var_{\ep}^{2}$,
we find a constant $d=d_{\ep}\in(1,2)$ such that
\begin{equation}
\int_{\pa B_{d}(y_{\ep})}\left(\ep^{2}|\na\var_{\ep}|^{2}+\var_{\ep}^{2}\right)\le\|\var_{\ep}\|_{\ep}^{2}.\label{eq: choose the radii}
\end{equation}
See Lemma 4.5 of \cite{Cao-Li-Luo-2015} for similar applications.
Also, for $d$ defined as above, it follows from (\ref{eq: Properties of U})
that
\begin{equation}
\ep^{2}\int_{\pa B_{d}(y_{\ep})}|\na U_{\ep,y_{\ep}}|^{2}=O(e^{-\si_{1}/\ep})\label{eq: exponential decay on the sphere}
\end{equation}
holds for some $\si_{1}>0$ independent of $\ep$. By an elementary
inequality, we have
\[
\int_{\pa B_{d}(y_{\ep})}|\na u_{\ep}|^{2}\le2\int_{\pa B_{d}(y_{\ep})}|\na U_{\ep,y_{\ep}}|^{2}+2\int_{\pa B_{d}(y_{\ep})}|\na\var_{\ep}|^{2}.
\]
Hence, for the constant $d$ chosen as above, we deduce
\begin{equation}
\ep^{2}\int_{\pa B_{d}(y_{\ep})}|\na u_{\ep}|^{2}=O(\|\var_{\ep}\|_{\ep}^{2}).\label{eq: inequality on gradient}
\end{equation}
Now we can improve the estimate for the asymptotic behavior of $y_{\ep}$
with respect to $\ep$.

\begin{proposition} \label{prop: rate of concentration} Assume that
$V$ satisfies (V1), (V2) and (V3). Let $u_{\ep}=U_{\ep,y_{\ep}}+\var_{\ep}$
be a solution derived as in Theorem \ref{thm: main result-existence}.
Then
\begin{eqnarray*}
|y_{\ep}|=o(\ep) &  & \text{as }\ep\to0.
\end{eqnarray*}
\end{proposition}
\begin{proof}
To analyze the asymptotic behavior of $y_{\ep}$ with respect to $\ep$,
we apply the Pohozaev-type identity (\ref{eq: Pohozaev}) to $u=u_{\ep}$
with $\Om=B_{d}(y_{\ep})$, where $d\in(1,2)$ is chosen as in (\ref{eq: choose the radii}).
Note that $d$ is possibly dependent on $\ep$. We get
\begin{equation}
\int_{B_{d}(y_{\ep})}\frac{\pa V}{\pa x_{i}}(U_{\ep,y_{\ep}}+\var_{\ep})^{2}=:\sum_{i=1}^{3}I_{i}\label{eq: 2.1}
\end{equation}
with
\[
\begin{aligned} & I_{1}=\left(\ep^{2}a+\ep b\int_{\R^{3}}|\na u_{\ep}|^{2}\right)\int_{\pa B_{d}(y_{\ep})}\left(|\na u_{\ep}|^{2}\nu_{i}-2\frac{\pa u_{\ep}}{\pa\nu}\frac{\pa u_{\ep}}{\pa x_{i}}\right),\\
 & I_{2}=\int_{\pa B_{d}(y_{\ep})}Vu_{\ep}^{2}\nu_{i}\text{ and }I_{3}=-\frac{2}{p+1}\int_{\pa B_{d}(y_{\ep})}u_{\ep}^{p+1}\nu_{i}.
\end{aligned}
\]
We estimate each side of (\ref{eq: 2.1}) as follows.

From (\ref{eq: size of the solution}) we get
\[
\ep^{2}a+\ep b\int_{\R^{3}}|\na u_{\ep}|^{2}=O(\ep^{2}).
\]
Thus, from (\ref{eq: inequality on gradient}) we deduce $I_{1}=O(\|\var_{\ep}\|_{\ep}^{2})$.
Using similar arguments and choosing a suitable $d$ if necessary,
we also get $I_{2}=O(\|\var_{\ep}\|_{\ep}^{2})$ and $I_{3}=O(\|\var_{\ep}\|_{\ep}^{p+1})$.
Hence
\begin{equation}
\sum_{i=1}^{3}I_{i}=O(\|\var_{\ep}\|_{\ep}^{2}).\label{eq: estimate of RHS of 2.1}
\end{equation}

To estimate the left hand side of (\ref{eq: 2.1}), notice that
\begin{equation}
\int_{B_{d}(y_{\ep})}\frac{\pa V}{\pa x_{i}}(U_{\ep,y_{\ep}}+\var_{\ep})^{2}=\int_{B_{d}(y_{\ep})}\frac{\pa V}{\pa x_{i}}U_{\ep,y_{\ep}}^{2}+O(\ep^{\frac{3}{2}}\|\var_{\ep}\|_{\ep})+O(\|\var_{\ep}\|_{\ep}^{2}).\label{eq: 2.2}
\end{equation}
By the assumption (V3), we deduce, for each $i=1,2,3$,
\begin{equation}
\begin{aligned}\int_{B_{d}(y_{\ep})}\frac{\pa V}{\pa x_{i}}U_{\ep,y_{\ep}}^{2} & =mc_{i}\int_{B_{d}(y_{\ep})}|x_{i}|^{m-2}x_{i}U_{\ep,y_{\ep}}^{2}+O\left(\int_{B_{d}(y_{\ep})}|x|^{m}U_{\ep,y_{\ep}}^{2}\right)\\
 & =mc_{i}\ep^{3}\int_{B_{\frac{d}{\ep}}(0)}|\ep z_{i}+y_{\ep,i}|^{m-2}(\ep z_{i}+y_{\ep,i})+O\left(\ep^{3}\left(\ep^{m}+|y_{\ep}|^{m}\right)\right)\\
 & =mc_{i}\ep^{3}\int_{\R^{3}}|\ep z_{i}+y_{\ep,i}|^{m-2}(\ep z_{i}+y_{\ep,i})U^{2}+O\left(\ep^{3}\left(\ep^{m}+|y_{\ep}|^{m}\right)\right).
\end{aligned}
\label{eq: 2.3}
\end{equation}
We have used (\ref{eq: Properties of U}) in the last two lines in
the above. (\ref{eq: 2.2}) and (\ref{eq: 2.3}) gives
\begin{equation}
\begin{aligned}\int_{B_{d}(y_{\ep})}\frac{\pa V}{\pa x_{i}}(U_{\ep,y_{\ep}}+\var_{\ep})^{2} & =mc_{i}\ep^{3}\int_{\R^{3}}|\ep z_{i}+y_{\ep,i}|^{m-2}(\ep z_{i}+y_{\ep,i})U^{2}\\
 & \quad+O\left(\ep^{\frac{3}{2}}\|\var_{\ep}\|_{\ep}+\|\var_{\ep}\|_{\ep}^{2}+\ep^{3}\left(\ep^{m}+|y_{\ep}|^{m}\right)\right).
\end{aligned}
\label{eq: 2.4}
\end{equation}

Since $c_{i}\neq0$ by assumption (V3), combining (\ref{eq: 2.1})-(\ref{eq: 2.4})
we deduce
\[
\ep^{3}\int_{\R^{3}}|\ep z_{i}+y_{\ep,i}|^{m-2}(\ep z_{i}+y_{\ep,i})U^{2}=O\left(\ep^{\frac{3}{2}}\|\var_{\ep}\|_{\ep}+\|\var_{\ep}\|_{\ep}^{2}+\ep^{3}\left(\ep^{m}+|y_{\ep}|^{m}\right)\right).
\]
By Proposition \ref{prop: reduction map} and (V3),
\[
\|\var_{\ep}\|_{\ep}=O\left(\ep^{3/2}(\ep^{m-\tau}+|y_{\ep}|^{m(1-\tau)})\right).
\]
Thus,
\begin{equation}
\int_{\R^{3}}|\ep z_{i}+y_{\ep,i}|^{m-2}(\ep z_{i}+y_{\ep,i})U^{2}=O\left(\ep^{m-\tau}+|y_{\ep}|^{m(1-\tau)}\right).\label{eq: the 1st estimate}
\end{equation}

On the other hand, let $m^{\ast}=\min(m,2)$. We have
\begin{equation}
\begin{aligned}|y_{\ep,i}|^{m} & \le|\ep z_{i}+y_{\ep,i}|^{m}-m|\ep z_{i}+y_{\ep,i}|^{m-2}(\ep z_{i}+y_{\ep,i})\ep z_{i}\\
 & \quad+C\left(|\ep z_{i}+y_{\ep,i}|^{m-m^{\ast}}|\ep z_{i}|^{m^{\ast}}+|\ep z_{i}|^{m}\right)\\
 & \le m|\ep z_{i}+y_{\ep,i}|^{m-2}(\ep z_{i}+y_{\ep,i})y_{\ep,i}+C\left(|\ep z_{i}|^{m}+|y_{\ep,i}|^{m-m^{\ast}}|\ep y_{i}|^{m^{\ast}}\right),
\end{aligned}
\label{eq: elementary iinequ.}
\end{equation}
by the following elementary inequality: for any $e,f\in\R$ and $m>1$,
there holds
\[
\left||e+f|^{m}-|e|^{m}-m|e|^{m-2}ef\right|\le C\left(|e|^{m-m^{\ast}}|f|^{m^{\ast}}+|f|^{m}\right)
\]
for some $C>0$ depending only on $m$. So, multiplying (\ref{eq: elementary iinequ.})
by $U^{2}$ on both sides and integrate over $\R^{3}$. We get
\[
|y_{\ep,i}|^{m}\int_{\R^{3}}U^{2}\le m\int_{\R^{3}}|\ep z_{i}+y_{\ep,i}|^{m-2}(\ep z_{i}+y_{\ep,i})y_{\ep,i}U^{2}+O\left(\ep^{m}+|y_{\ep}|^{m-m^{\ast}}\ep^{m^{\ast}}\right)
\]
for each $i$. Applying (\ref{eq: the 1st estimate}) to the above
estimate yields
\[
|y_{\ep}|^{m}=O\left(\left(\ep^{m-\tau}+|y_{\ep}|^{m(1-\tau)}\right)|y_{\ep}|+\ep^{m}+|y_{\ep}|^{m-m^{\ast}}\ep^{m^{\ast}}\right).
\]
Recall that $m\tau<1$. Using $\ep$-Young inequality
\begin{eqnarray*}
XY\le\de X^{m}+\de^{-\frac{m}{m-1}}Y^{\frac{m}{m-1}}, &  & \forall\:\de,X,Y>0,
\end{eqnarray*}
we deduce
\[
|y_{\ep}|=O(\ep).
\]

We have to prove that $|y_{\ep}|=o(\ep)$. Assume, on the contrary,
that there exist $\ep_{k}\to0$ and $y_{\ep_{k}}\to0$ such that $y_{\ep_{k}}/\ep_{k}\to A\in\R^{3}$
with $A\neq0$. Then (\ref{eq: the 1st estimate}) gives
\[
\int_{\R^{3}}\left|z+\frac{y_{\ep_{k}}}{\ep_{k}}\right|^{m-2}\left(z_{i}+\frac{y_{\ep_{k}}}{\ep_{k}}\right)U^{2}=O(\ep^{m-\tau}),
\]
Taking limit in the above gives
\[
\int_{\R^{3}}\left|z+A\right|^{m-2}\left(z+A\right)U^{2}(z)=0.
\]
However, since $U=U(|z|)$ is strictly decreasing with respect to
$|z|$, we infer that $A=0$. We reach a contradiction. The proof
is complete.
\end{proof}
As a consequence of Proposition \ref{prop: rate of concentration}
and the assumption (V3), we infer that
\begin{equation}
\|\var_{\ep,y}\|_{\ep}=O\left(\ep^{\frac{3}{2}+m(1-\tau)}\right).\label{eq: improved estimate for the correction term}
\end{equation}
Here we can take $\tau$ so small that $m(1-\tau)>1$ since $m>1$.

\section{Proof of Theorem \ref{thm: uniqueness} \label{sec: Proof-of-uniqueness-Theorem }}

This section is devoted to proving Theorem \ref{thm: uniqueness}.
We use a contradiction argument as that of Cao, Li and Luo \cite{Cao-Li-Luo-2015}.
Assume $u_{\ep}^{(i)}=U_{\ep,y_{\ep}^{(i)}}+\var_{\ep}^{(i)}$, $i=1,2$,
are two distinct solutions derived as in Theorem \ref{thm: main result-existence}.
By (\ref{eq: decay estimates}), $u_{\ep}^{(i)}$ are bounded functions
in $\R^{3}$, $i=1,2$. Set
\[
\xi_{\ep}=\frac{u_{\ep}^{(1)}-u_{\ep}^{(2)}}{\|u_{\ep}^{(1)}-u_{\ep}^{(2)}\|_{L^{\wq}(\R^{3})}}
\]
and set
\[
\bar{\xi}_{\ep}(x)=\xi_{\ep}(\ep x+y_{\ep}^{(1)}).
\]
It is clear that
\[
\|\bar{\xi}_{\ep}\|_{L^{\wq}(\R^{3})}=1.
\]
Moreover, by the remark following (\ref{eq: decay estimates}), there
holds
\begin{eqnarray}
\bar{\xi}_{\ep}(x)\to0 &  & \text{as }|x|\to\wq\label{eq: vanishing at infinity}
\end{eqnarray}
uniformly with respect to sufficiently small $\ep>0$. We will reach
a contradiction by showing that $\|\bar{\xi}_{\ep}\|_{L^{\wq}(\R^{3})}\to0$
as $\ep\to0$. In view of (\ref{eq: vanishing at infinity}), it suffices
to show that for any fixed $R>0$,
\begin{eqnarray}
\|\bar{\xi}_{\ep}\|_{L^{\wq}(B_{R}(0))}\to0 &  & \text{as }\ep\to0.\label{eq: local uniform convergence}
\end{eqnarray}
To this end, we will establish a series of results. First we have

\begin{proposition}\label{prop: asym of normalized solutions}There
holds
\[
\|\xi_{\ep}\|_{\ep}=O(\ep^{3/2}).
\]
\end{proposition}
\begin{proof}
Since both $u_{\ep}^{(i)}$, $i=1,2$, are assumed to be solutions
to Eq. (\ref{eq: Kirchhoff}), we obtain that
\begin{equation}
\begin{aligned} & -\left(\ep^{2}a+\ep b\int_{\R^{3}}|\na u_{\ep}^{(1)}|^{2}\right)\De\xi_{\ep}+V\xi_{\ep}\\
 & \qquad=\ep b\left(\int_{\R^{3}}\na(u_{\ep}^{(1)}+u_{\ep}^{(2)})\cdot\na\xi_{\ep}\right)\De u_{\ep}^{(2)}+C_{\ep}(x)\xi_{\ep},
\end{aligned}
\label{eq: comparison eq. 1}
\end{equation}
and that
\begin{equation}
\begin{aligned} & -\left(\ep^{2}a+\ep b\int_{\R^{3}}|\na u_{\ep}^{(2)}|^{2}\right)\De\xi_{\ep}+V\xi_{\ep}\\
 & \qquad=\ep b\left(\int_{\R^{3}}\na(u_{\ep}^{(1)}+u_{\ep}^{(2)})\cdot\na\xi_{\ep}\right)\De u_{\ep}^{(1)}+C_{\ep}(x)\xi_{\ep},
\end{aligned}
\label{eq: comparison eq. 2}
\end{equation}
where
\[
C_{\ep}(x)=p\int_{0}^{1}\left(tu_{\ep}^{(1)}(x)+(1-t)u_{\ep}^{(2)}(x)\right)^{p-1}.
\]
Adding (\ref{eq: comparison eq. 1}) and (\ref{eq: comparison eq. 2})
together gives
\begin{equation}
\begin{aligned} & -\left(2\ep^{2}a+\ep b\int_{\R^{3}}|\na u_{\ep}^{(1)}|^{2}+|\na u_{\ep}^{(2)}|^{2}\right)\De\xi_{\ep}+2V\xi_{\ep}\\
 & \qquad=\ep b\left(\int_{\R^{3}}\na(u_{\ep}^{(1)}+u_{\ep}^{(2)})\cdot\na\xi_{\ep}\right)\De\left(u_{\ep}^{(1)}+u_{\ep}^{(2)}\right)+2C_{\ep}(x)\xi_{\ep}.
\end{aligned}
\label{eq: symmetric comparision eq.}
\end{equation}
Multiply $\xi_{\ep}$ on both sides of (\ref{eq: symmetric comparision eq.})
and integrate over $\R^{3}$. By throwing away the terms containing
$b$, we get
\[
\|\xi_{\ep}\|_{\ep}^{2}\le\int_{\R^{3}}C_{\ep}\xi_{\ep}^{2}\D x.
\]
On the other hand, note that $C_{\ep}\le C\sum_{i=1}^{2}(u_{\ep}^{(i)})^{p-1}$.
This implies
\[
\begin{aligned}\int_{\R^{3}}C_{\ep}\xi_{\ep}^{2} & \le C\sum_{i=1}^{2}\int_{\R^{3}}\left(u_{\ep}^{(i)}\right)^{p-1}\xi_{\ep}^{2}\\
 & \le C\sum_{i=1}^{2}\left(\int_{\R^{3}}\left(u_{\ep}^{(i)}\right)^{6}\right)^{\frac{p-1}{6}}\left(\int_{\R^{3}}\left(\xi_{\ep}^{2}\right)^{\frac{6}{7-p}}\right)^{\frac{7-p}{6}}\\
 & \le C\sum_{i=1}^{2}\|\na u_{\ep}^{(i)}\|_{L^{2}(\R^{3})}^{p-1}\left(\int_{\R^{3}}\xi_{\ep}^{2}\right)^{\frac{7-p}{6}}\\
 & =O(\ep^{\frac{p-1}{2}})\|\xi_{\ep}\|_{\ep}^{(7-p)/3}.
\end{aligned}
\]
In the last inequality we have used the fact that $\|\xi_{\ep}\|_{L^{\wq}(\R^{3})}=1$
and (\ref{eq: size of the solution}).

Therefore,
\[
\|\xi_{\ep}\|_{\ep}^{2}=O(\ep^{\frac{p-1}{2}})\|\xi_{\ep}\|_{\ep}^{(7-p)/3},
\]
which implies the desired estimate. The proof is complete.
\end{proof}
Next we study the asymptotic behavior of $\bar{\xi}_{\ep}$.

\begin{proposition}\label{prop: convergence of xi_epsilon} Let $\bar{\xi}_{\ep}=\xi_{\ep}(\ep x+y_{\ep}^{(1)})$.
There exist $d_{i}\in\R$, $i=1,2,3$, such that (up to a subsequence)
\begin{eqnarray*}
\bar{\xi}_{\ep}\to\sum_{i=1}^{3}d_{i}\pa_{x_{i}}U &  & \text{in }C_{\loc}^{1}(\R^{3})
\end{eqnarray*}
as $\ep\to0$.\end{proposition}
\begin{proof}
It is straightforward to deduce from (\ref{eq: comparison eq. 1})
that $\bar{\xi}_{\ep}$ solves
\begin{equation}
\begin{aligned} & -\left(a+\ep^{-1}b\int_{\R^{3}}|\na u_{\ep}^{(1)}|^{2}\right)\De\bar{\xi}_{\ep}+V(\ep x+y_{\ep}^{(1)})\bar{\xi}_{\ep}\\
 & \qquad=\ep^{-1}b\left(\int_{\R^{3}}\na(u_{\ep}^{(1)}+u_{\ep}^{(2)})\cdot\na\xi_{\ep}\right)\De\left(u_{\ep}^{(2)}(\ep x+y_{\ep}^{(1)})\right)+C_{\ep}(\ep x+y_{\ep}^{(1)})\bar{\xi}_{\ep}.
\end{aligned}
\label{eq: normalized equa.}
\end{equation}
For convenience, we introduce
\begin{eqnarray*}
\bar{u}_{\ep}^{(i)}(x)=u_{\ep}^{(i)}(\ep x+y_{\ep}^{(1)}) & \text{and} & \bar{\var}_{\ep}^{(i)}=\var_{\ep}^{(i)}(\ep x+y_{\ep}^{(1)})
\end{eqnarray*}
for $i=1,2$. Then, we have
\begin{equation}
\ep^{-1}\int_{\R^{3}}|\na u_{\ep}^{(1)}|^{2}=\int_{\R^{3}}|\na\bar{u}_{\ep}^{(1)}|^{2}\label{eq: 5.0}
\end{equation}
and
\begin{equation}
\ep^{-1}b\left(\int_{\R^{3}}\na(u_{\ep}^{(1)}+u_{\ep}^{(2)})\cdot\na\xi_{\ep}\right)=b\int_{\R^{3}}\na\left(\bar{u}_{\ep}^{(1)}+\bar{u}_{\ep}^{(2)}\right)\cdot\na\bar{\xi}_{\ep},\label{eq: 5.1}
\end{equation}
which are uniformly bounded for $\ep$ by (\ref{eq: uniform boundedness})
and by (\ref{eq: 5.3}) below. Moreover, we have
\begin{equation}
\int_{\R^{3}}\left|\na\bar{\var}_{\ep}^{(i)}\right|^{2}=\ep^{-3}O\left(\left\Vert \bar{\var}_{\ep}^{(i)}\right\Vert _{\ep}^{2}\right)=O(\ep^{2m(1-\tau)})\label{eq: 5.2}
\end{equation}
by (\ref{eq: improved estimate for the correction term}), and
\begin{equation}
\int_{\R^{3}}\left|\na\bar{\xi}_{\ep}\right|^{2}=\ep^{-1}\int_{\R^{3}}|\na\xi|^{2}=O(1)\label{eq: 5.3}
\end{equation}
by Proposition \ref{prop: asym of normalized solutions}.

Thus, in view of $\|\bar{\xi}_{\ep}\|_{L^{\wq}(\R^{3})}=1$ and (\ref{eq: boundedness estimate for laplacian})
and estimates in the above, the elliptic regularity theory implies
that $\bar{\xi}_{\ep}$ is locally uniformly bounded with respect
to $\ep$ in $C_{\loc}^{1,\be}(\R^{3})$ for some $\be\in(0,1)$.
As a consequence, we assume (up to a subsequence) that
\begin{eqnarray*}
\bar{\xi}_{\ep}\to\bar{\xi} &  & \text{in }C_{\loc}^{1}(\R^{3}).
\end{eqnarray*}
We claim that $\bar{\xi}\in{\rm Ker}\L$, that is,
\begin{equation}
-\left(a+b\int_{\R^{3}}|\na U|^{2}\right)\De\bar{\xi}-2b\left(\int_{\R^{3}}\na U\cdot\na\bar{\xi}\right)\De U+\bar{\xi}=pU^{p-1}\bar{\xi}.\label{eq: blow-up equations}
\end{equation}
Then $\bar{\xi}=\sum_{i=1}^{3}d_{i}\pa_{x_{i}}U$ follows from Theorem
\ref{prop: LPX-2016} for some $d_{i}\in\R$ ($i=1,2,3$), and thus
Proposition \ref{prop: convergence of xi_epsilon} is proved.

To deduce (\ref{eq: blow-up equations}), we only need to show that
(\ref{eq: blow-up equations}) is the limiting equation of Eq. (\ref{eq: normalized equa.}).
It follows from (\ref{eq: 5.0}) and (\ref{eq: 5.2}) that
\begin{equation}
\begin{aligned}\ep^{-1}b\int_{\R^{3}}|\na u_{\ep}^{(1)}|^{2}-b\int_{\R^{3}}|\na U|^{2} & =b\int_{\R^{3}}\left(|\na\bar{u}_{\ep}^{(1)}|^{2}-|\na U|^{2}\right)\\
 & =b\int_{\R^{3}}\left(\left|\na U+\na\bar{\var}_{\ep}^{(1)}\right|^{2}-|\na U|^{2}\right)\\
 & =O(\ep^{m(1-\tau)})\to0.
\end{aligned}
\label{eq: 5.4}
\end{equation}
Similarly, we deduce from (\ref{eq: 5.1}) (\ref{eq: 5.2}) and (\ref{eq: 5.3})
that
\[
\begin{aligned}\int_{\R^{3}}\na\left(\bar{u}_{\ep}^{(1)}+\bar{u}_{\ep}^{(2)}-2U\right)\cdot\na\bar{\xi}_{\ep} & =\int_{\R^{3}}\na\left(U\left(x+(y_{\ep}^{(1)}-y_{\ep}^{(2)})/\ep\right)-U\right)\cdot\na\bar{\xi}_{\ep}\\
 & \quad+\int_{\R^{3}}\na\left(\bar{\var}_{\ep}^{(1)}+\bar{\var}_{\ep}^{(2)}\right)\cdot\na\bar{\xi}_{\ep}\\
 & =o(1),
\end{aligned}
\]
and that, for any $\Phi\in C_{0}^{\wq}(\R^{3})$,
\[
\begin{aligned}\int_{\R^{3}}\na\left(\bar{u}_{\ep}^{(2)}-U\right)\cdot\na\Phi & =\int_{\R^{3}}\na\left(U\left(x+(y_{\ep}^{(1)}-y_{\ep}^{(2)})/\ep\right)-U\right)\cdot\na\Phi\\
 & \quad+\int_{\R^{3}}\na\bar{\var}_{\ep}^{(2)}\cdot\na\Phi\\
 & \to0.
\end{aligned}
\]
Here, we have used Proposition \ref{prop: rate of concentration},
which implies $(y_{\ep}^{(1)}-y_{\ep}^{(2)})/\ep\to0$ as $\ep\to0$.
Combining the above two formulas and (\ref{eq: 5.1}) and $\bar{\xi}_{\ep}\to\bar{\xi}$
in $C_{\loc}^{1}(\R^{3})$, we conclude that
\begin{equation}
\frac{b}{\ep}\left(\int_{\R^{3}}\na(u_{\ep}^{(1)}+u_{\ep}^{(2)})\cdot\na\xi_{\ep}\right)\De\left(u_{\ep}^{(2)}(\ep x+y_{\ep}^{(1)})\right)\to2b\left(\int_{\R^{3}}\na U\cdot\na\bar{\xi}\right)\De U\label{eq: 5.5}
\end{equation}
in $H^{-1}(\R^{3})$.

Also, similar to Lemma 3.2 of Cao, Li and Luo \cite{Cao-Li-Luo-2015},
we have for any $\Phi\in C_{0}^{\wq}(\R^{3})$,
\begin{equation}
\int_{\R^{3}}C_{\ep}(\ep x+y_{\ep}^{(1)})\bar{\xi}_{\ep}\Phi-p\int_{\R^{3}}U^{p-1}\bar{\xi}_{\ep}\Phi=o(1).\label{eq: 5.6}
\end{equation}

Finally, combining (\ref{eq: 5.4}) (\ref{eq: 5.5}) (\ref{eq: 5.6}),
we obtain (\ref{eq: blow-up equations}). The proof is complete.
\end{proof}
Now we prove (\ref{eq: local uniform convergence}) by showing the
following lemma.

\begin{lemma} \label{lem: vanishing limit function}Let $d_{i}$
be defined as in Proposition \ref{prop: convergence of xi_epsilon}.
Then
\begin{eqnarray*}
d_{i}=0 &  & \text{for }i=1,2,3.
\end{eqnarray*}
 \end{lemma}
\begin{proof}
We use the Pohozaev-type identity (\ref{eq: Pohozaev}) to prove this
lemma. Apply (\ref{eq: Pohozaev}) to $u_{\ep}^{(1)}$ and $u_{\ep}^{(2)}$
with $\Om=B_{d}(y_{\ep}^{(1)})$, where $d$ is chosen in the same
way as that of Proposition \ref{prop: rate of concentration}. We
obtain
\[
\begin{aligned} & \int_{B_{d}(y_{\ep}^{(1)})}\frac{\pa V}{\pa x_{i}}\left(\left(u_{\ep}^{(1)}\right)^{2}-\left(u_{\ep}^{(2)}\right)^{2}\right)\\
 & =\left(\ep^{2}a+\ep b\int_{\R^{3}}|\na u_{\ep}^{(1)}|^{2}\right)\int_{\pa B_{d}(y_{\ep}^{(1)})}\left(|\na u_{\ep}^{(1)}|^{2}\nu_{i}-2\frac{\pa u_{\ep}^{(1)}}{\pa\nu}\frac{\pa u_{\ep}^{(1)}}{\pa x_{i}}\right)\\
 & \quad-\left(\ep^{2}a+\ep b\int_{\R^{3}}|\na u_{\ep}^{(2)}|^{2}\right)\int_{\pa B_{d}(y_{\ep}^{(1)})}\left(|\na u_{\ep}^{(2)}|^{2}\nu_{i}-2\frac{\pa u_{\ep}^{(2)}}{\pa\nu}\frac{\pa u_{\ep}^{(2)}}{\pa x_{i}}\right)\\
 & \quad+\int_{\pa B_{d}(y_{\ep}^{(1)})}V(x)\left(\left(u_{\ep}^{(1)}\right)^{2}-\left(u_{\ep}^{(2)}\right)^{2}\right)\nu_{i}\\
 & \quad-\frac{2}{p+1}\int_{\pa B_{d}(y_{\ep}^{(1)})}\left(\left(u_{\ep}^{(1)}\right)^{p+1}-\left(u_{\ep}^{(2)}\right)^{p+1}\right)\nu_{i}.
\end{aligned}
\]
In terms of $\xi_{\ep}$, we get
\begin{equation}
\begin{aligned} & \int_{B_{d}(y_{\ep}^{(1)})}\frac{\pa V}{\pa x_{i}}\left(u_{\ep}^{(1)}+u_{\ep}^{(2)}\right)\xi_{\ep}\\
 & =\left(\ep^{2}a+\ep b\int_{\R^{3}}|\na u_{\ep}^{(1)}|^{2}\right)\int_{\pa B_{d}(y_{\ep}^{(1)})}\left(|\na u_{\ep}^{(1)}|^{2}\nu_{i}-2\frac{\pa u_{\ep}^{(1)}}{\pa\nu}\frac{\pa u_{\ep}^{(1)}}{\pa x_{i}}\right)\\
 & \quad-\left(\ep^{2}a+\ep b\int_{\R^{3}}|\na u_{\ep}^{(2)}|^{2}\right)\int_{\pa B_{d}(y_{\ep}^{(1)})}\left(|\na u_{\ep}^{(2)}|^{2}\nu_{i}-2\frac{\pa u_{\ep}^{(2)}}{\pa\nu}\frac{\pa u_{\ep}^{(2)}}{\pa x_{i}}\right)\\
 & \quad+\int_{\pa B_{d}(y_{\ep}^{(1)})}V\left(u_{\ep}^{(1)}+u_{\ep}^{(2)}\right)\xi_{\ep}\nu_{i}-2\int_{\pa B_{d}(y_{\ep}^{(1)})}A_{\ep}\xi_{\ep}\nu_{i},
\end{aligned}
\label{eq: 12}
\end{equation}
where $A_{\ep}(x)=\int_{0}^{1}(tu_{\ep}^{(1)}(x)+(1-t)u_{\ep}^{(2)}(x))^{p}$.

We estimate (\ref{eq: 12}) term by term. Note that
\[
\ep^{2}a+\ep b\int_{\R^{3}}|\na u_{\ep}^{(i)}|^{2}=O(\ep^{2})
\]
 holds by (\ref{eq: size of the solution}) for each $i=1,2$. Moreover,
by similar arguments as that of Proposition \ref{prop: rate of concentration},
we have
\[
\int_{\pa B_{d}(y_{\ep}^{(1)})}|\na u_{\ep}^{(i)}|^{2}=O(\|\na\var_{\ep}^{(i)}\|_{L^{2}(\R^{3})}^{2}).
\]
Thus, by (\ref{eq: improved estimate for the correction term}),
\[
\begin{aligned} & \sum_{i=1}^{2}\left(\ep^{2}a+\ep b\int_{\R^{3}}|\na u_{\ep}^{(i)}|^{2}\right)\int_{\pa B_{d}(y_{\ep}^{(1)})}\left||\na u_{\ep}^{(i)}|^{2}\nu_{i}-2\frac{\pa u_{\ep}^{(i)}}{\pa\nu}\frac{\pa u_{\ep}^{(i)}}{\pa x_{i}}\right|\\
 & \quad=\sum_{i=1}^{2}O\left(\ep^{2}\|\|\na\var_{\ep}^{(i)}\|_{L^{2}(\R^{3})}^{2}\right)\\
 & \quad=O(\ep^{3+2m(1-\tau)}).
\end{aligned}
\]
Also, similar to that of Cao, Li and Luo \cite{Cao-Li-Luo-2015},
we have
\[
\int_{\pa B_{d}(y_{\ep}^{(1)})}V(x)\left(u_{\ep}^{(1)}+u_{\ep}^{(2)}\right)\xi_{\ep}\nu_{i}=O(\ep^{3+m(1-\tau)})
\]
and
\[
\int_{\pa B_{d}(y_{\ep}^{(1)})}A_{\ep}\xi_{\ep}\nu_{i}=O(\ep^{(3+m(1-\tau))p}).
\]
Hence we conclude that
\begin{equation}
\text{the RHS of }(\ref{eq: 12})=O(\ep^{3+m(1-\tau)}).\label{eq: 12.right}
\end{equation}

Next we estimate the left hand side of (\ref{eq: 12}). We have
\begin{equation}
\begin{aligned} & \int_{B_{d}(y_{\ep}^{(1)})}\frac{\pa V}{\pa x_{i}}\left(u_{\ep}^{(1)}+u_{\ep}^{(2)}\right)\xi_{\ep}\\
 & =mc_{i}\int_{B_{d}(y_{\ep}^{(1)})}|x_{i}|^{m-2}x_{i}\left(u_{\ep}^{(1)}+u_{\ep}^{(2)}\right)\xi_{\ep}+O\left(\int_{B_{d}(y_{\ep}^{(1)})}|x_{i}|^{m}\left(u_{\ep}^{(1)}+u_{\ep}^{(2)}\right)\xi_{\ep}\right).
\end{aligned}
\label{eq: 5.5.0}
\end{equation}
Observe that
\[
\begin{aligned} & mc_{i}\int_{B_{d}(y_{\ep}^{(1)})}|x_{i}|^{m-2}x_{i}\left(u_{\ep}^{(1)}+u_{\ep}^{(2)}\right)\xi_{\ep}\\
 & =mc_{i}\ep^{3}\int_{B_{\frac{d}{\ep}}(0)}|\ep y_{i}+y_{\ep,i}^{(1)}|^{m-2}\left(\ep y_{i}+y_{\ep,i}^{(1)}\right)\left(U(y)+U\left(y+\frac{y_{\ep}^{(1)}-y_{\ep}^{(2)}}{\ep}\right)\right)\bar{\xi}_{_{\ep}}\\
 & \quad+mc_{i}\int_{B_{d}(y_{\ep}^{(1)})}|x_{i}|^{m-2}x_{i}\left(\var_{\ep}^{(1)}+\var_{\ep}^{(2)}\right)\xi_{\ep}.
\end{aligned}
\]
Since $U$ decays exponentially at infinity and $y_{\ep}^{(i)}=o(\ep)$,
using Proposition \ref{prop: convergence of xi_epsilon} we deduce
\begin{equation}
\begin{aligned} & mc_{i}\ep^{3}\int_{B_{\frac{d}{\ep}}(0)}|\ep y_{i}+y_{\ep,i}^{(1)}|^{m-2}\left(\ep y_{i}+y_{\ep,i}^{(1)}\right)\left(U(y)+U\left(y+\frac{y_{\ep}^{(1)}-y_{\ep}^{(2)}}{\ep}\right)\right)\bar{\xi}_{_{\ep}}\\
 & =2mc_{i}\ep^{m+2}\sum_{j=1}^{3}d_{j}\int_{\R^{3}}|y_{i}|^{m-2}y_{i}U(y)\pa_{x_{j}}U+o(\ep^{m+2})\\
 & =D_{i}d_{i}\ep^{m+2}+o(\ep^{m+2}),
\end{aligned}
\label{eq: 5.5.1}
\end{equation}
where
\begin{equation}
D_{i}=2mc_{i}\int_{\R^{3}}|y|^{m-2}y_{i}U(y)\pa_{x_{i}}U\neq0.\label{eq: D-i}
\end{equation}
In the last equality of (\ref{eq: 5.5.1}), we used the fact that
$U$ is a radially symmetric function. On the other hand, by H\"older's inequality,
(\ref{eq: improved estimate for the correction term}) and Proposition
\ref{prop: asym of normalized solutions}, we have
\begin{equation}
\begin{aligned}mc_{i}\int_{B_{d}(y_{\ep}^{(1)})}|x_{i}|^{m-2}x_{i}\left(\var_{\ep}^{(1)}+\var_{\ep}^{(2)}\right)\xi_{\ep} & =\sum_{i=1}^{2}O\left(\int_{\R^{3}}|\var_{\ep}^{(i)}||\xi_{\ep}|\right)\\
 & =\sum_{i=1}^{2}O(\|\var_{\ep}^{(i)}\|_{\ep}\|\xi_{\ep}\|_{\ep})\\
 & =O(\ep^{3+m(1-\tau)}).
\end{aligned}
\label{eq: 5.5.2}
\end{equation}
Therefore, from (\ref{eq: 5.5.1}) and (\ref{eq: 5.5.2}), we deduce
\begin{equation}
mc_{i}\int_{B_{d}(y_{\ep}^{(1)})}|x_{i}|^{m-2}x_{i}\left(u_{\ep}^{(1)}+u_{\ep}^{(2)}\right)\xi_{\ep}\D x=D_{i}d_{i}\ep^{m+2}+o(\ep^{m+2}),\label{eq: 5.5.3}
\end{equation}
 with $D_{i}\neq0$ given by (\ref{eq: D-i}). Similar arguments gives
\begin{equation}
O\left(\int_{B_{d}(y_{\ep}^{(1)})}|x_{i}|^{m}\left(u_{\ep}^{(1)}+u_{\ep}^{(2)}\right)\xi_{\ep}\D x\right)=O(\ep^{m+3}).\label{eq: 5.5.4}
\end{equation}
Hence, combining (\ref{eq: 5.5.3}) and (\ref{eq: 5.5.4}), we obtain
\begin{equation}
\text{the RHS of }(\ref{eq: 12})=D_{i}d_{i}\ep^{m+2}+o(\ep^{m+2}).\label{eq: 12.left}
\end{equation}

At last, this lemma follows from (\ref{eq: 12.right}) and (\ref{eq: 12.left}).
The proof is complete.
\end{proof}
Now we can prove Theorem \ref{thm: uniqueness}.

\begin{proof}[Proof of Theorem \ref{thm: uniqueness}] If there exist
two distinct solutions $u_{\ep}^{(i)}$, $i=1,2$, then by setting
$\xi_{\ep}$ and $\bar{\xi}_{\ep}$ as above, we find that
\[
\|\bar{\xi}_{\ep}\|_{L^{\wq}(\R^{3})}=1
\]
 by assumption, and that
\begin{eqnarray*}
\|\bar{\xi}_{\ep}\|_{L^{\wq}(\R^{3})}=o(1) &  & \text{as }\ep\to0
\end{eqnarray*}
by (\ref{eq: vanishing at infinity}) and (\ref{eq: local uniform convergence}).
We reach a contradiction. The proof is complete.\end{proof}

We close this paper by remarking that the proof of Theorem \ref{thm: uniqueness}
implies the following slightly more general uniqueness result, which
allows $c_{i}$ has different signs for $i=1,2,3$ in the assumption
(V3).

\begin{theorem}\label{thm: generalized uniqueness result} Let $a,b>0$
and $1<p<5$. Assume that $V$ satisfies (V1) and (V3). If $u_{\ep}^{(i)}=U_{\ep,y_{\ep}^{(i)}}+\var_{\ep,y_{\ep}^{(i)}}^{(i)}$,
$i=1,2$, are two solutions to Eq. (\ref{eq: Kirchhoff}) satisfying
$y_{\ep}\to0$ and
\[
\left\Vert \var_{\ep,y_{\ep}^{(i)}}^{(i)}\right\Vert _{\ep}\le\ep^{\frac{3}{2}}\left(\ep^{\al-\tau}+\left(V(y_{\ep}^{(i)})-V(0)\right)^{1-\tau}\right)
\]
for some $\tau>0$ sufficiently small. Then
\[
u_{\ep}^{(1)}\equiv u_{\ep}^{(2)}.
\]
Moreover, writing $u_{\ep}=U_{\ep,y_{\ep}}+\var_{\ep}$ as the unique
solution, there holds
\[
y_{\ep}=o(\ep),
\]
\[
\left\Vert \var_{\ep}\right\Vert _{\ep}=O\left(\ep^{\frac{3}{2}+m(1-\tau)}\right).
\]
\end{theorem}

\emph{Acknowledgment}. Part of this work was finished when Xiang was studying in the University of Jyvaskyla, Finland. He would like to thank the financial support of  the Academy of Finland.  

\appendix

\section{The unperturbed problem}

Let $U$ be the unique positive radial solution of Eq. (\ref{eq: limiting equ.})
(see Proposition \ref{prop: LPX-2016}). Then, for any $\ep>0$ and
$y\in\R^{3}$, the function
\begin{eqnarray*}
U_{\ep,y}(x)\equiv U\left(\frac{x-y}{\ep}\right), &  & x\in\R^{3}
\end{eqnarray*}
are the unique positive solutions to equation
\begin{eqnarray*}
-\left(\ep^{2}a+\ep b\int_{\R^{3}}|\na u|^{2}\D x\right)\De u+u=u^{p} &  & \text{in }\R^{3}.
\end{eqnarray*}
We also recall that in the assumption (V2) we assumed that $x_{0}=0$,
$r_{0}=10$ and $V(0)=1$.

\begin{proposition}\label{Prop: asymptotics of perturbation of U}
Assume that $V$ satisfies (V1) and (V2). Let $y\in B_{1}(0)$. Then
, for $\ep>0$ sufficiently small, we have
\[
I_{\ep}(U_{\ep,y})=A\ep^{3}+B\ep^{3}\left(V(y)-V(0)\right)+O(\ep^{3+\al}),
\]
 where
\[
A=\frac{1}{2}\int_{\R^{3}}\left(a|\na U|^{2}+U^{2}\right)+\frac{b}{4}\left(\int_{\R^{3}}|\na U|^{2}\right)^{2}-\frac{1}{p+1}\int_{\R^{3}}U^{p+1}
\]
and
\[
B=\frac{1}{2}\int_{\R^{3}}U^{2}.
\]
\end{proposition}
\begin{proof}
By direct computation, we obtain
\[
\begin{aligned}I_{\ep}(U_{\ep,y}) & =\frac{1}{2}\int_{\R^{3}}\left(\ep^{2}a|\na U_{\ep,y}|^{2}+V(x)U_{\ep,y}^{2}\right)+\frac{\ep b}{4}\left(\int_{\R^{3}}|\na U_{\ep,y}|^{2}\right)^{2}-\frac{1}{p+1}\int_{\R^{3}}U_{\ep,y}^{p+1}\\
 & =A\ep^{3}+\frac{1}{2}\int_{\R^{3}}(V(x)-V(0))U_{\ep,y}^{2}\\
 & =A\ep^{3}+B\ep^{3}\left(V(y)-V(0)\right)+\frac{1}{2}\int_{\R^{3}}\left(V(x)-V(y)\right)U_{\ep,y}^{2},
\end{aligned}
\]
where $A,B$ are given as in the result.

Split the last term into two terms:
\[
\int_{\R^{3}}\left(V(x)-V(y)\right)U_{\ep,y}^{2}=\int_{B_{1}(y)}\left(V(x)-V(y)\right)U_{\ep,y}^{2}+\int_{\R^{3}\backslash B_{1}(y)}\left(V(x)-V(y)\right)U_{\ep,y}^{2}.
\]
Using the assumption (V2) and $y\in B_{1}(0)$, we deduce
\[
\int_{B_{1}(y)}|V(x)-V(y)|U_{\ep,y}^{2}=O(\ep^{3+\al}).
\]
Using the boundedness of $V$ and the exponential decay of $U$, we
deduce
\[
\int_{\R^{3}\backslash B_{1}(y)}|V(x)-V(y)|U_{\ep,y}^{2}=O(\ep^{3+\al}).
\]
 Combining above estimates gives the desired estimates. The proof
is complete.
\end{proof}

\section{Proof of Lemma \ref{lem: error estimates}}

This section is devoted to the proof of Lemma \ref{lem: error estimates}.
Recall that $R_{\ep}$ is defined as in (\ref{eq: 2rd order reminder term}),
which gives
\begin{equation}
R_{\ep}(\var)=A_{1}(\var)-A_{2}(\var),\label{eq: error term}
\end{equation}
where
\[
A_{1}(\var)=\frac{b\epsilon}{4}\left(\left(\int_{\R^{3}}|\na\var|^{2}\right)^{2}+4\int_{\R^{3}}|\na\var|^{2}\int_{\R^{3}}\na U_{\ep,y}\cdot\na\var\right)
\]
and
\[
A_{2}(\var)=\frac{1}{p+1}\int_{\R^{3}}\left(\left(U_{\ep,y}+\var\right)_{+}^{p+1}-U_{\ep,y}^{p+1}-(p+1)U_{\ep,y}^{p}\var-\frac{p(p+1)}{2}U_{\ep,y}^{p-1}\var^{2}\right).
\]
Use $R_{\ep}^{(i)}$ to denote the $i$th derivative of $R_{\ep}$,
and also use similar notations for $A_{1}$ and $A_{2}$. By direct
computations, we deduce that, for any $\var,\psi\in H_{\ep}$,
\[
\langle R_{\ep}^{(1)}(\var),\psi\rangle=\langle A_{1}^{(1)}(\var),\psi\rangle-\langle A_{2}^{(1)}(\var),\psi\rangle
\]
where
\[
\begin{aligned}\langle A_{1}^{(1)}(\var),\psi\rangle & =b\ep\left(\int_{\R^{3}}|\na\var|^{2}\int_{\R^{3}}\na\var\cdot\na\psi+\int_{\R^{3}}|\na\var|^{2}\int_{\R^{3}}\na U_{\ep,y}\cdot\na\psi\right)\\
 & \quad+2b\ep\int_{\R^{3}}\na U_{\ep,y}\cdot\na\var\int_{\R^{3}}\na\var\cdot\na\psi
\end{aligned}
\]
and

\[
\langle A_{2}^{(1)}(\var),\psi\rangle=\int_{\R^{3}}\left((U_{\ep,y}+\var)_{+}^{p}\psi-U_{\ep,y}^{p}\psi-pU_{\ep,y}^{p-1}\var\psi\right).
\]
We also deduce, for any $\var,\psi,\xi\in H_{\ep}$, that
\[
\langle R_{\ep}^{(2)}(\var)[\psi],\xi\rangle=\langle A_{1}^{(2)}(\var)[\psi],\xi\rangle-\langle A_{2}^{(2)}(\var)[\psi],\xi\rangle,
\]
where
\[
\begin{aligned}\langle A_{1}^{(2)}(\var)[\psi],\xi\rangle & =b\ep\left(2\int_{\R^{3}}\na\var\cdot\na\psi\int_{\R^{3}}\na\var\cdot\na\xi+\int_{\R^{3}}|\na\var|^{2}\int_{\R^{3}}\na\xi\cdot\na\psi\right)\\
 & \quad+2b\ep\left(\int_{\R^{3}}\na\var\cdot\na\psi\int_{\R^{3}}\na U_{\ep,y}\cdot\na\xi+\int_{\R^{3}}\na U_{\ep,y}\cdot\na\psi\int_{\R^{3}}\na\var\cdot\na\xi\right)\\
 & \quad+2b\ep\int_{\R^{3}}\na U_{\ep,y}\cdot\na\var\int_{\R^{3}}\na\xi\cdot\na\psi
\end{aligned}
\]
and
\[
\langle A_{2}^{(2)}(\var)[\psi],\xi\rangle=\int_{\R^{3}}\left(p\left(U_{\ep,y}+\var\right)_{+}^{p-1}\psi\xi-pU_{\ep,y}^{p-1}\psi\xi\right).
\]
Now we prove Lemma \ref{lem: error estimates}.

\begin{proof}[Proof of Lemma \ref{lem: error estimates}] First, we
estimate $A_{1}(\var)$, $A_{1}^{\prime}(\var)$ and $A_{1}^{\prime\prime}(\var)$.
Notice that
\[
\|\na U_{\ep,y}\|_{L^{2}(\R^{3})}=C_{0}\ep^{1/2}
\]
with $C_{0}=\|\na U\|_{L^{2}(\R^{3})}$, and that
\begin{eqnarray*}
\|\na\var\|_{L^{2}(\R^{3})}\le C_{1}\ep^{-1}\|\var\|_{\ep}, &  & \var\in H_{\ep}
\end{eqnarray*}
holds for some $C_{1}>0$ independent of $\ep$. Combining above two
estimates together with H\"older's inequality yields
\[
\int_{\R^{3}}|\na\var\cdot\na\psi|\int_{\R^{3}}|\na U_{\ep,y}\cdot\na\xi|\le C\ep^{-5/2}
\]
and that
\[
\int_{\R^{3}}|\na\var\cdot\na\psi|\int_{\R^{3}}|\na\eta\cdot\na\xi|\le C\ep^{-4}
\]
for all $\var,\psi,\eta,\xi\in H_{\ep}$. These estimates imply that
\[
|A_{1}^{(i)}(\var)|\le Cb\ep^{-\frac{3}{2}}\left(1+\ep^{-\frac{3}{2}}\|\var\|_{\ep}\right)\|\var\|_{\ep}^{3-i}
\]
 for some constant $C>0$ independent of $\ep$.

Next we estimate $A_{2}^{(i)}(\var)$ (the $i$th derivative of $A_{2}(\var)$)
for $i=0,1,2$. We consider the case $1<p\le2$ first.

To estimate $A_{2}(\var)$, we apply the following elementary inequality:
for any $e,f\in\R$, there exists $C_{1}(p)>0$ depending only on
$p$, so that
\[
\left|(e+f)_{+}^{p+1}-e_{+}^{p+1}-(p+1)e_{+}^{p}f-\frac{p(p+1)}{2}e_{+}^{p-1}f^{2}\right|\le C_{1}(p)|f|^{p+1}.
\]
Then there holds
\[
|A_{2}(\var)|\le C\int_{\R^{3}}|\var|^{p+1}\le C\ep^{-\frac{3(p-1)}{2}}\|\var\|_{\ep}^{p+1},
\]
where we have used (\ref{eq: epsilon-Sobolev inequality}) to derive
the second term.

To estimate $A_{2}^{(1)}(\var)$, we apply the following elementary
inequality: for any $e,f\in\R$, there exists $C_{2}(p)>0$ depending
only on $p$, so that
\[
\left|(e+f)_{+}^{p}-e_{+}^{p}-pe_{+}^{p-1}f\right|\le C_{2}(p)|f|^{p}.
\]
Then there holds
\[
|\langle A_{2}^{(1)}(\var),\psi\rangle|\le C_{p}\int_{\R^{3}}|\var|^{p}|\psi|\le C\ep^{-\frac{3(p-1)}{2}}\|\var\|_{\ep}^{p}\|\psi\|_{\ep},
\]
where we have used (\ref{eq: epsilon-Sobolev inequality}) to derive
the second term. This gives
\[
\|A_{2}^{(1)}(\var)\|\le C\ep^{-\frac{3(p-1)}{2}}\|\var\|_{\ep}^{p}.
\]

To estimate $A_{2}^{(2)}(\var)$, we apply the following elementary
inequality: for any $e,f\in\R$, there exists $C_{3}(p)>0$ depending
only on $p$, so that
\[
\left|(e+f)_{+}^{p-1}-e_{+}^{p-1}\right|\le C_{3}(p)|f|^{p-1}.
\]
Then there holds
\[
|\langle A_{2}^{\prime\prime}(\var)[\psi],\xi\rangle|\le C_{3}(p)\int_{\R^{3}}|\var|^{p-1}|\psi||\xi|\le C\ep^{-\frac{3(p-1)}{2}}\|\var\|_{\ep}^{p-1}\|\psi\|_{\ep}\|\xi\|_{\ep}.
\]
where we have used H\"older's inequality and (\ref{eq: epsilon-Sobolev inequality})
to derive the second term. This gives
\[
\|A_{2}^{(2)}(\var)\|\le C\ep^{-\frac{3(p-1)}{2}}\|\var\|_{\ep}^{p-1}.
\]

Combining the above estimates yields the result in Lemma \ref{lem: error estimates}
in the case $1<p\le2$.

In the case $2<p<5$, we can estimate $A_{2}^{(i)}(\var)$ similarly
as above. So we only point out the following elementary inequalities
that are needed. For any $e,f\in\R$, there exist $\tilde{C}_{i}(p)>0$
($i=1,2,3$) such that
\[
\left|(e+f)_{+}^{p+1}-e_{+}^{p+1}-(p+1)e_{+}^{p}f-\frac{p(p+1)}{2}e_{+}^{p-1}f^{2}\right|\le\tilde{C}_{1}(p)(|e|^{p-2}+|f|^{p-2})|f|^{3},
\]

\[
\left|(e+f)_{+}^{p}-e_{+}^{p}-pe_{+}^{p-1}f\right|\le\tilde{C}_{2}(p)(|e|^{p-2}+|f|^{p-2})|f|^{2}
\]
and
\[
\left|(e+f)_{+}^{p-1}-e_{+}^{p-1}\right|\le\tilde{C}_{3}(p)(|e|^{p-2}+|f|^{p-2})|f|.
\]
The proof of Lemma \ref{lem: error estimates} is complete.\end{proof}

\end{document}